\documentclass[12pt]{amsart}
\usepackage{amssymb, amsmath, amscd, mathtools, mathrsfs}

\usepackage[dvipsnames]{xcolor}
\usepackage{tikz}
\usetikzlibrary{decorations.markings}
\usetikzlibrary{shapes,snakes}
\usepackage[all,cmtip]{xy}
\usepackage{stmaryrd}
\usetikzlibrary{shapes.geometric}
\usepackage{color}
\usepackage[utf8]{inputenc}
\usepackage[T1]{fontenc}
\usepackage{fullpage}
\usepackage[normalem]{ulem}
\usepackage{amsmath,amscd,mathrsfs}
\usepackage{verbatim}
\usepackage[all,cmtip]{xy}
\usepackage[colorlinks=true, citecolor=blue, linkcolor=red,pagebackref, hyperindex]{hyperref}
\usepackage{mathtools}

\newtheorem{theoremx}{Theorem}

\newtheorem{Theorem}{Theorem}[section]

\newtheorem{corollaryx}[theoremx]{Corollary}
\newtheorem{Potential Theorem}[Theorem]{Potential Theorem}
\newtheorem{Lemma}[Theorem]{Lemma}
\newtheorem{Corollary}[Theorem]{Corollary}
\newtheorem{Proposition}[Theorem]{Proposition}
\theoremstyle{definition}
\newtheorem{Example}[Theorem]{Example}

\newtheorem{Definition}[Theorem]{Definition}

\newtheorem{Notation}[Theorem]{Notation}

\theoremstyle{remark} 
\newtheorem{Remark}[Theorem]{Remark}

\DeclareMathOperator{\chara}{char}

\DeclareMathOperator{\rank}{rank}

\DeclareMathOperator{\frk}{frk}

\DeclareMathOperator{\s}{s}

\DeclareMathOperator{\mult}{mult}

\def\m{\mathfrak{m}}

\def\Z{\mathbb{Z}}
\def\R{\mathbb{R}}
\def\Q{\mathbb{Q}}

\def\N{\mathbb{N}}

\def\ds{\displaystyle}

\renewcommand{\geq}{\geqslant} 
\renewcommand{\leq}{\leqslant} 

\newcommand{\CC}{\mathbb{C}}
\newcommand{\ZZ}{\mathbb{Z}}

\newcommand{\ps}[1]{\llbracket{#1} \rrbracket}

\newcommand{\kk}{\Bbbk}

\begin{document}

\title{F-signature function of quotient singularities}
\author{Alessio Caminata}
\address{Institut de Matem\`{a}tica, Universitat de Barcelona \\ Gran Via de les Corts Catalanes 585, 08007 Barcelona, Spain}
\email{caminata@ub.edu}
\thanks{The first author is supported by European Union's Horizon 2020 research and innovation programme under grant agreement No 701807.}

\author{Alessandro De Stefani}
\address{Department of Mathematics, University of Nebraska, 203 Avery Hall, Lincoln, NE 68588}
\email{adestefani2@unl.edu}

\begin{abstract}
 We study the shape of the F-signature function of a $d$-dimensional quotient singularity $\kk\ps{x_1,\ldots,x_d}^G$, and we show that it is a quasi-polynomial.
We prove that the second coefficient is always zero and we describe the other coefficients in terms of invariants of the finite acting group $G\subseteq {\rm Gl}(d,\kk)$. When $G$ is cyclic, we obtain more specific formulas for the coefficients of the quasi-polynomial, which allow us to compute the general form of the function in several examples.
\end{abstract}

\maketitle
\section{Introduction}

Let $(R,\m,\kk)$ be a commutative complete Noetherian local domain of characteristic $p>0$, and assume that the residue field $\kk = R/\m$ is perfect. For a positive integer $e$, let $F^e:R \to R$ denote the $e$-th iterate of the Frobenius endomorphism on $R$. The map $F^e$ can be identified with the $R$-module inclusion $R \hookrightarrow R^{1/p^e}$, where $R^{1/p^e}$ is the ring obtained by adjoining $p^e$-th roots of elements in $R$. The main object of study of this article is the {\it F-signature function of $R$}, that is, the function
\[
\xymatrixrowsep{1mm}
\xymatrixcolsep{1mm}
\xymatrix{
FS:& \N \ar[rrr] &&& \N \\ 
&e \ar[rrr] &&& \frk_R(R^{1/p^e}),
}
\]
where $\frk_R(R^{1/p^e})$ denotes the maximal rank of a free $R$-summand of $R^{1/p^e}$ or, equivalently, the maximal rank of a free $R$-module $P$ for which there is a surjection $R^{1/p^e} \to P\to 0$. 

The F-signature function has been introduced by Smith and Van den Bergh, in the context of rings with finite F-representation type \cite{SmithVDB}. Even though this function has several interesting properties, most of the efforts have been devoted to studying its leading term, called the F-signature of $R$, and denoted $\s(R)$ (see Section \ref{Section_background} for more precise definitions). Despite being a coarser invariant, $\s(R)$ already encodes a significant amount of information about the ring and its singularities. For example, $R$ is regular if and only if $\s(R)=1$ \cite{HunekeLeuschke}, and $R$ is strongly F-regular if and only if $\s(R)>0$ \cite{AberbachLeuschke}. However, $\s(R)$ is typically very hard to compute explicitly, and it is known only in a few sporadic cases. Moreover, the techniques that allow to determine $\s(R)$ often do not allow to compute of the whole F-signature function. Therefore, even less is known about $FS(e)$, with a few very special exceptions (for instance, see \cite{Brinkmann}, or \cite[Example 7]{SinghFSignature}).

 Another function that can be defined in the same setup is the Hilbert-Kunz function $e \mapsto HK(e) = \ell_R(R/\m^{[p^e]})$, where $\ell_R$ denotes the length of an $R$-module, and $\m^{[p^e]}$ is the ideal generated by the elements $r^{p^e}$, for $r \in \m$.
The Hilbert-Kunz function was first investigated by Kunz in \cite{Kunz} and \cite{F-finExc}. In \cite{Monsky}, Monsky showed that $HK(e)=e_{HK}(R)p^{de}+O(p^{(d-1)e})$, where $e_{HK}(R)$ is a positive real number called Hilbert-Kunz multiplicity and $d$ is the Krull dimension of $R$. The main connection with the F-signature function can be best stated when $R$ is a Gorenstein singularity with minimal multiplicity, in which case $HK(e) = \ell_R(R/(x_1,\ldots,x_d))p^{de} - FS(e)$ for all $e\in\mathbb{N}$ \cite{HunekeLeuschke}. Here, $x_1,\ldots,x_d$ denotes any minimal reduction of the maximal ideal $\m$. In the case when $R$ is Gorenstein and minimal multiplicity, a knowledge of the function $HK(e)$ therefore leads to that of the F-signature function $FS(e)$. As the F-signature function, the Hilbert-Kunz function is also quite mysterious, and known only for very specific classes of rings. Among other results in this direction, see \cite{Brenner}, \cite{Brinkmann}, \cite{Kurano2}, \cite{HanMonsky}, \cite{F-finExc}, \cite{Kurano1}, \cite{MillerSwanson}, \cite{RobinsonSwanson}. 

In the effort of understanding the shape of the Hilbert-Kunz function, the question of whether there exists a ``second coefficient'' for $HK(e)$ has caught the attention of several researchers. One says that $HK(e)$ has a second coefficient if there exists $\beta\in\mathbb{R}$ such that $HK(e)=e_{HK}(R)p^{de}+\beta p^{(d-1)e}+O(p^{(d-2)e})$. Huneke, McDermott, and Monsky \cite{HunekeMcDermottMonsky} prove that, if $R$ is excellent, normal, and $F$-finite, then this is the case.
Chan and Kurano \cite{Kurano2} prove that the same result holds if one replaces normal with regular in codimension one.
Brenner \cite{Brenner} shows that, for standard graded normal domains of dimension two over an algebraically closed field, the second coefficient equals zero. In \cite{Kurano1}, Kurano proves that the same conclusion holds for F-finite $\Q$-Gorenstein local rings with algebraically closed residue field. 

For the F-signature function, it is known that $FS(e)=s(R)p^{de}+O(p^{(d-1)e})$ (see \cite{Tucker2012}, \cite{PolTuc}).
In their recent work, Polstra and Tucker ask whether a second coefficient for the F-signature function exists as well \cite[Question 7.4]{PolTuc}. We thank Polstra for pointing out to us that this is known to be true for some classes of rings, including rings that are $\Q$-Gorenstein on the punctured spectrum and affine semigroup rings, as a consequence of the existence of a second coefficient for Hilbert-Kunz functions with respect to $\m$-primary ideals \cite{HunekeMcDermottMonsky}. Using this approach, Brinkmann \cite{Brinkmann} computes the F-signature function of 2-dimensional ADE singularities, and shows that the second coefficient exists, and it is equal to zero.   
In this article, we prove that the same result holds for the larger class of $d$-dimensional quotient singularities (see Theorem \ref{theoremA}).

Throughout, $\kk$ denotes an algebraically closed field, and $G \subseteq \mathrm{Gl}(d,\kk)$ is a finite group, that acts linearly on $S=\kk\ps{x_1,\ldots,x_d}$. We assume that the characteristic of $\kk$ does not divide $|G|$, and we let $R=S^G$ be the ring of invariants under this action. We say that an element $g$ of $G$ is a $c$-pseudoreflection if, when viewed as an element of $\mathrm{Gl}(d,\kk)$, it has eigenvalue $1$ with multiplicity $c$, and $d-c$ eigenvalues different from $1$. In particular, the only $d$-pseudoreflection is the identity. In what follows, we assume that the group $G$ is small, that is, $G$ contains no $(d-1)$-pseudoreflections.

There is a known connection between the F-signature of $R$ and the acting group $G$: in our assumptions, $\s(R) = \frac{1}{|G|}$ \cite[Theorem 4.2]{WatanabeYoshida}. In fact, even deeper connections can be established for the generalized F-signature of certain modules \cite{HashimotoNakajima}, even in a more general setup \cite{HashimotoSymonds}. We further develop the relation established in \cite{WatanabeYoshida}, giving a description of the F-signature function of $R$ in terms of $c$-pseudoreflections.  

\begin{theoremx}[see Theorem \ref{theorem-Fsignaturefunctionquotient} and Proposition \ref{prop-Fsignaturefuncionquotient}] \label{theoremA} Let $\kk$ be an algebraically closed field of positive characteristic $p$, and $G$ be a finite small subgroup of $\mathrm{Gl}(d,\kk)$ such that $p\nmid |G|$. Let $S=\kk\llbracket x_1,\dots,x_d\rrbracket$ be a power series ring, and let $R=S^G$ be the ring of invariants of $S$ under the action of $G$. The F-signature function of $R$ is a quasi-polynomial in $p^e$:
\[
FS(e) = \varphi_dp^{de}+\varphi_{d-1}p^{(d-1)e}+\cdots+\varphi_{1}p^e+\varphi_{0}.
	\]
For $0 \leq c \leq d$, $\varphi_c=\varphi_{c}(e)$ is a function that takes values in $\Q$, is bounded, and periodic of period at most $|G|-1$. Moreover:
\begin{enumerate}
\item $\varphi_c$ is identically zero if and only if $G$ does not contain any $c$-pseudoreflections. 
\item  If $p^e \equiv 1$ modulo $|G|$, then $\varphi_c(e) = \frac{|G_c|}{|G|}$.
\end{enumerate}
In particular, we have that $\varphi_d(e)=\frac{1}{|G|}$, and $\varphi_{d-1}(e)=0$ for all $e\in\mathbb{N}$.
\end{theoremx}
We remark that, when $G$ is Abelian, the fact that $FS(e)$ is a quasi-polynomial with rational coefficients can be deduced from \cite{Bruns} and \cite{VonKorff}, since in this case $R$ is toric.

As quotient singularities have finite F-representation type \cite{SmithVDB}, our methods actually yield more general formulas for the multiplicities $\mult(M_\alpha,R^{1/p^e})$, where the modules $M_\alpha$ run over the irreducible $R$-modules that appear in a direct sum decomposition of $R^{1/p^e}$, for $e\in\mathbb{N}$. The multiplicity functions include the F-signature function, since $M_0=R$, and thus $FS(e) = \mult(M_0,R^{1/p^e})$. In analogy with $FS(e)$, the aforementioned generalized F-signature of a module $M_\alpha$ is the leading coefficient $\varphi_d^{(\alpha)}$ of the quasi-polynomial $\mult(M_\alpha,R^{1/p^e})$. In this sense, Theorem \ref{theorem-Fsignaturefunctionquotient} generalizes the main result of \cite{HashimotoNakajima}.

In the second part of the article, we focus on the case when the group $G$ is cyclic of order $n$. Viewing a generator $g \in G$ as an element of $\mathrm{Gl}(d,\kk)$, one can assume that $g$ is represented by a diagonal matrix, with $n$-th roots of unity on the diagonal. We can then associate to $g$ a $d$-uple $(t_1,\ldots,t_d)$ that records the multiplicative order of the elements on the diagonal. For every $J \subseteq \{1,\ldots,d\}$, we set $g_J$ to be the greatest common divisor of $n$, together with the integers $\{t_j \ : \ j \in J\}$. For instance, $g_{\{1,\ldots,d\}} = \gcd(t_1,\ldots,t_d,n)$, while $g_{\{1\}} = \gcd(t_1,n)$. 

Our second main result is a formula to explicitly compute the F-signature function of $R$ in terms of the integers $g_J$, and functions $e \mapsto \theta_J(e)$ which count the number of solutions of certain congruences (see Notation \ref{notation_psi} for more details). Let $\Gamma_i$ be the set of subsets of $\{1,2,\ldots,d\}$ of cardinality $i$, and set $\psi_i = \sum_{J \in \Gamma_i} g_J\theta_J$. The functions $\psi_i = \psi_i(e)$ are also bounded and periodic, of period at most $|G|-1$.

\begin{theoremx}[see Theorems \ref{theorem-Fsignaturecyclic} and \ref{theorem-peggiodelcasomonomiale}] \label{theoremB}
In the setup of Theorem \ref{theoremA}, assume further that $G$ is cyclic of order $n$. For all $e\in\mathbb{N}$, write $p^e=kn+r_e$, where $0 < r_e < n$. With the notation introduced above, the functions $\varphi_c$ can then be expressed as
\[
\ds  \varphi_c(e) = \frac{1}{n} \left[\sum_{i=c}^d (-1)^{i-c}{i \choose c}\psi_i(e)r_e^{i-c}\right].
\]
\end{theoremx}
As for Theorem \ref{theoremA}, also the formulas of Theorem \ref{theoremB} can be generalized to similar formulas for the functions $\varphi_c^{(\alpha)}(e)$, which are the coefficients of the multiplicity functions $\mult(M_\alpha,R^{1/p^e})$ (see Theorem \ref{theorem-Fsignaturecyclic}).

As a direct consequence of Theorem \ref{theoremB}, we obtain an explicit description of the F-signature function of Veronese rings up to a bounded periodic function $\theta_\emptyset$, defined in Notation \ref{notation_psi}. Recall that the (complete) $d$-dimensional Veronese ring of order $n$ over a field $\kk$ is the ring $R=\kk\ps{x_1,\ldots,x_d}^G$, where $G=\ZZ/(n)$, and a generator $g \in G$ is identified with the matrix ${\rm diag}(\lambda,\ldots,\lambda) \in {\rm Gl}(d,\kk)$, where $\lambda$ is a primitive $n$-th root of unity in $\kk$. Alternatively, $R$ can be viewed as the completion at the irrelevant maximal ideal of the $\kk$-subalgebra of $\kk[x_1,\ldots,x_d]$ generated by the monomials of degree $n$ in the variables $x_1,\ldots,x_d$.
\begin{corollaryx}[see Corollary \ref{coroll_gcd1}]
Let $R$ be a $d$-dimensional Veronese ring of order $n$ over an algebraically closed field of characteristic $p>0$. For $e\in\mathbb{N}$, write $p^e=kn + r_e$. The F-signature function of $R$ is
\[
\ds FS(e) = \frac{p^{de}-r_e^d}{n} + \theta_\emptyset,
\]
where $\theta_\emptyset$ is the number of integral $d$-uples $(a_1,\ldots,a_d)$, contained inside the $d$-dimensional cube $[0,r_e-1]^d$, that satisfy $a_1+\ldots + a_d \equiv 0$ modulo $n$.  In particular, if $r_e=1$, then
\[
\ds FS(e) = \frac{p^{de}-1}{n} + 1.
\]
\end{corollaryx}

This paper is structured as follows: in Section \ref{Section_background}, we recall the main definitions and results concerning the F-signature function and Auslander's correspondence, that we use extensively throughout the article. In Section \ref{Section_quotient_sing}, we study the F-signature function of quotient singularities, and prove Theorem \ref{theoremA}. In Section \ref{section:cyclicquotient} we focus on the cyclic case to obtain Theorem \ref{theoremB}, and deduce a formula for Veronese rings. Finally, in Section \ref{section:examples} we provide several examples, to explicitly illustrate how Theorem \ref{theoremA} and Theorem \ref{theoremB} allow to compute the F-signature function of some specific quotient singularities.

\section{Background} \label{Section_background}
\subsection{F-signature function}
Let $R$ be a commutative Noetherian ring of prime characteristic $p>0$. For a positive integer $e$, let $F^e:R \to R$ denote the $e$-th iterate of the Frobenius endomorphism on $R$, that is, the map that raises every element of $R$ to its $p^e$-th power. Given a finitely generated $R$-module $M$, we denote by  $^e\!M$ the module $M$, whose $R$-module structure is pulled back via $F^e$. More explicitly, for $^e\!m_1,^e\!m_2 \in \!^e\!M$ and $r \in R$ we have
\[
\ds ^e\!m_1+\!^e\!m_2 = \!^e\!(m_1+m_2) \ \ \mbox{ and } \ \ r \cdot \!^e\!m_1 = \!^e\!(r^{p^e}m_1).
\]
When the ring $R$ is reduced, the Frobenius endomorphism $F^e:R\rightarrow\ \! R$ can be identified with the natural inclusion $R\hookrightarrow R^{1/p^e}$, where $R^{1/p^e}$ is the ring obtained by adding $p^e$-th roots of elements in $R$. In particular, $^e\!R$ can be identified with $R^{1/p^e}$.

Throughout, we assume that $(R,\m,\kk)$ is a complete local domain with perfect residue field.  Let $K$ be the fraction field of $R$. By the rank of a finitely generated $R$-module $M$, we mean the dimension of the $K$-vector space $M \otimes_R K$. We let $\frk_R(R^{1/p^e})$ denote the maximal rank of a free $R$-module $P$ for which there is a surjection $R^{1/p^e} \to P \to 0$. 

We now introduce the main object of study of this article. 

\begin{Definition}[Smith-Van den Bergh, Huneke-Leuschke] \label{Defn_FSig}
Let $(R,\m,\kk)$ be a complete local domain with perfect residue field.
The {\it F-signature function of $R$} is defined as
\[
\xymatrixrowsep{1mm}
\xymatrixcolsep{1mm}
\xymatrix{
FS:& \N \ar[rrr] &&& \N \\ 
&e \ar[rrr] &&& \frk_R(R^{1/p^e}).
}
\]
\end{Definition}

The F-signature function has been introduced and first studied by Smith and Van den Bergh, with main focus on rings with finite F-representation type \cite{SmithVDB}. See the end of the section for a more precise definition. Successively, in \cite{HunekeLeuschke}, Huneke and Leuschke focused on an asymptotic normalized version of this function: they defined the F-signature of $R$ as the limit $\s(R) = \lim_{e \to \infty} \frac{FS(e)}{p^{de}}$, where $d$ is the Krull dimension of $R$. It is easy to see that $0 \leq \s(R) \leq 1$ always holds, but the convergence of such a limit is far from trivial. The existence of the F-signature in full generality was a major open problem, until Tucker gave a proof in  \cite{Tucker2012}. In joint work with Polstra and Yao, the second author generalized the definition of F-signature to a more general setup, where the ring does not need to be local \cite{DSPY}.

\begin{Remark} In Definition \ref{Defn_FSig}, the assumption that $R$ is a complete domain and $\kk$ is perfect can be greatly weakened. However, the type of rings we will investigate in this article are of this form. Therefore, we do not provide the most general definition here. 
\end{Remark} 

\par Let $(R,\mathfrak{m},\kk)$ be a complete Noetherian local ring with perfect residue field. The category of finitely generated $R$-modules satisfies the Krull-Remak-Schmidt property. It follows that every finitely generated $R$-module can be uniquely decomposed (up to isomorphism) as a direct sum of indecomposable finitely generated $R$-modules.
In our running assumptions, $R$ is F-finite. This means that, for each $e\in\mathbb{N}$, the module $R^{1/p^e}$ is finitely generated, hence a direct sum of finitely generated indecomposable $R$-modules. We say that $R$ has \emph{finite F-representation type} (FFRT for short) if there exists a finite set $\mathcal{N}$ of indecomposable $R$-modules such that for every $e\in\mathbb{N}$ the $R$-module $R^{1/p^e}$ is isomorphic to a direct sum of elements of $\mathcal{N}$. 
In other words, if $R$ has FFRT and $\mathcal{N}=\{M_0=R,M_1,\dots,M_r\}$, then for all $e\in\mathbb{N}$ we can write 
\[
R^{1/p^e}\cong M_0^{c_{0,e}}\oplus M_1^{c_{1,e}}\cdots\oplus M_r^{c_{r,e}}
\]
for some uniquely determined integers $c_{0,e},\ldots,c_{r,e}$.
\begin{Notation}
For $\alpha\in\{0,\dots,r\}$, we denote $\mult(M_{\alpha},R^{1/p^e})=c_{\alpha,e}$ and we call it \emph{the multiplicity of $M_{\alpha}$ inside $R^{1/p^e}$}.	
In particular for $\alpha=0$, we have $M_0=R$ and the function $\mult(R,R^{1/p^e})=\frk_R(R^{1/p^e})=FS(e)$ is  the F-signature function of $R$.
\end{Notation}

\par The notion of FFRT and the functions $e\mapsto\mult(M_{\alpha},R^{1/p^e})$ were introduced by Smith and Van den Bergh \cite{SmithVDB}. They proved that if $R$ is strongly F-regular then the limit 
\[
\lim_{e \to \infty} \frac{\mult(M_{\alpha},R^{1/p^e})}{p^{de}}
\]
exists, and is strictly positive. Around the same time, Seibert \cite{Seibert} studied a similar problem and proved the existence of the previous limit, assuming that $R$ has finite Cohen-Macaulay type. 

\par As already pointed out for the F-signature function, in this article we are interested in studying the functions $\mult(M_{\alpha},R^{1/p^e})$, rather than the asymptotic behavior of \ $\frac{\mult(M_\alpha, R^{1/p^e})}{p^{de}}$.

\subsection{Non-modular representation theory in positive characteristic}
In this subsection, we recall some basic definitions and results on non-modular representation theory in positive characteristic.
\par We fix an algebraically closed field $\kk$ of positive characteristic $p$ and a finite subgroup $G\subseteq\mathrm{Gl}(d,\kk)$ such that $p\nmid |G|$.
When we say that $(V,\rho)$ is a $\kk$-representation of $G$, we will always mean a finite-dimensional  $\kk$-linear representation of $G$, i.e., a group homomorphism $\rho:G\rightarrow\mathrm{Gl}(V)$, where $V$ is a finite-dimensional $\kk$-vector space.
By abuse of notation, we will sometimes just call $V$ the representation, meaning that a map $\rho$ is given as well. The dimension of the representation is just the $\kk$-dimension of $V$.
Thanks to Mashke's theorem, the category of $\kk$-representations of $G$ has the Krull-Remak-Schmidt property, with the indecomposable objects being the irreducible representations. In other words, any representation $V$ can be uniquely decomposed (up to isomorphism)  as a direct sum of irreducible representations:
\begin{equation*}
V\cong V_0^{c_0}\oplus\cdots\oplus V_r^{c_r},
\end{equation*}
where $V_0,\dots,V_r$ are pairwise non-isomorphic irreducible representations.
\begin{Notation}
The natural number $c_i$ is called the \textit{multiplicity of} $V_i$ inside $V$, and we denote it by $\mult(V_i,V)=c_i$.
We will use the notation $V_0$ to denote the trivial representation of $G$ given by $g\in G\mapsto 1\in \mathrm{Gl}(1,\kk)=\kk^*$.	
\end{Notation}
Finally, we recall that the number of non-isomorphic irreducible $\kk$-representations of $G$ is finite and equal to the number of conjugacy classes of $G$.
 
\begin{Definition}[Frobenius twist]	Let $\kk$ be a perfect field, and $V$ be a $\kk$-vector space. For any positive integer $e$, we denote by $V^{1/p^e}=\{v^{1/p^e}: \ v\in V\}$ the $\kk^{1/p^e} = \kk$-vector space with sum and scalar multiplication given by
\[
\ds v_1^{1/p^e} + v_2^{1/p^e} = (v_1+v_2)^{1/p^e}, \mbox{ and} \ a\cdot v_1^{1/p^e} =(a^{p^e}v_1)^{1/p^e}
\]
for $a\in \kk$ and $v_1^{1/p^e}, v_2^{1/p^e} \in V^{1/p^e}$.  
If $V$ is a $\kk$-representation of a group $G$, then the composition $G\hookrightarrow \mathrm{Gl}(V)\xrightarrow{\Phi}\mathrm{Gl}(V^{1/p^e})$ shows that $V^{1/p^e}$ is also a representation of $G$, where $\Phi$ is given by $\Phi(g)(v^{1/p^e})=\ (gv)^{1/p^e}$, for $g\in G$, $v\in V$. We call this representation the \emph{$e$-th Frobenius twist} of $V$.
\end{Definition}	

\begin{Remark}
	Let $v_1,\dots,v_s$ be a basis of $V$, and assume that the representation $V$ of $G$ is given by a matrix $(f_{i,j}(g))$, where $f_{i,j}(g) \in \kk$ for all $g \in G$. Explicitly, this means that for $g \in G$ we have $g\cdot v_j=\sum_{i=1}^sf_{i,j}(g)v_j$.
	Since $\kk$ is algebraically closed, the elements $v_1^{1/p^e},\dots,v_s^{1/p^e}$ form a $\kk$-basis of $V^{1/p^e}$, and the matrix representation of the Frobenius twist $V^{1/p^e}$ is given by $\left(f_{i,j}(g)^{1/p^e}\right)$.
\end{Remark}
\begin{Remark} \label{Remark_order_roots}
Observe that, if $(f_{i,j}(g))$ is in diagonal form, then every element $f_{i,i}(g)$ that appears on the main diagonal is a primitive $m$-th root of unity in $\kk$, where $m$ divides the order of $g$ in $G$. Since $\kk$ is algebraically closed, and $p$ does not divide $m$, the map $(-)^{1/p^e}:\mu_m(\kk)\rightarrow\mu_m(\kk)$ is an isomorphism of groups, where $\mu_m(\kk)$ denotes the group of $m$-th roots of unity in $\kk$. In particular, $f_{i,i}(g)^{1/p^e}$ is also a primitive $m$-th root of unity in $\kk$.
\end{Remark}
\par We fix an isomorphism $\phi:\mu_{|G|}(\kk)\rightarrow\mu_{|G|}(\mathbb{C})$ between the groups of $|G|$-th roots of unity in $\kk$ and $|G|$-roots of unity in $\mathbb{C}$.
Let $(V,\rho)$ be a $\kk$-representation of $G$ of dimension $s\geq1$ and let $g$ be an element of $G$. Since $G$ is finite and $\kk$ is algebraically closed, the matrix $\rho(g)$ is diagonalizable in $\kk$. We denote by $\lambda_1,\dots,\lambda_s$ the eigenvalues of $\rho(g)$, counted with multiplicity. Observe that since $\mathrm{ord}_G(g)$ divides $|G|$, $\lambda_1,\dots,\lambda_s$ are elements of $\mu_{|G|}(\kk)$.

\begin{Definition}
	The \emph{Brauer character} or simply the \emph{character} of $(V,\rho)$ is the function $\chi_V:G\rightarrow\mathbb{C}$ given by $\chi_{V}(g)=\phi(\lambda_1)+\cdots+\phi(\lambda_s)$. 
\end{Definition}
We collect some properties of Brauer characters in the following proposition.

\begin{Proposition}\label{propertiesofcharacterprop}
	Let $V$ be a $\kk$-representations of $G$ with character $\chi_V$, and let $V_i$ be an irreducible $\kk$-representation of $G$ with character $\chi_{V_i}$.
	Then the following facts hold:
	\begin{enumerate}
		\item $\chi_V(\mathrm{Id}_G)=\dim_{\kk}V$, where $\mathrm{Id}_G$ is the identity of $G$;
		\item $\chi_V(g^{-1})=\overline{\chi_V(g)}$, the complex conjugate of $\chi_V(g)$, for every $g\in G$;
		\item The multiplicity of $V_i$ in $V$ is given by
		\[\mult(V_i,V)=\frac{1}{|G|}\sum_{g\in G}\overline{\chi_{V_i}(g)}\cdot\chi_V(g),
		\] where $\overline{\chi_{V_i}(g)}$ is the complex conjugate of $\chi_{V_i}(g)$.
	\end{enumerate} 
\end{Proposition}

\par We conclude with the following well-known definition.
\begin{Definition}\label{Def-pseudoreflection}
	An element $g\in G\subseteq\mathrm{Gl}(d,\kk)$ is called a \emph{pseudoreflection} if the fixed subspace $\{v\in \kk^d: gv=v\}$ has dimension $d-1$. The group $G$ is called \emph{small} if it does not contain any pseudoreflections.
\end{Definition}

We observe that, since $G$ is finite and $\kk$ algebraically closed, then $g\in G$ is a pseudo-reflection if and only if it has an eigenvalue $1$ of multiplicity $d-1$ and another eigenvalue $\lambda\neq1$ of multiplicity $1$.

\section{F-signature function of quotient singularities} \label{Section_quotient_sing}

Let $\kk$ be an algebraically closed field of positive characteristic $p$, and let $G$ be a finite small subgroup of $\mathrm{Gl}(d,\kk)$ such that $p\nmid |G|$.
We consider a power series ring $S=\kk\llbracket x_1,\dots,x_d\rrbracket$ over $\kk$. 
The group $G$ acts linearly on $S$ with the action on the variables $x_1,\dots,x_d$ given by matrix multiplication. This defines a unique $\kk$-representation of $G$ of dimension $d$, which is called \emph{fundamental representation} of $G$. We denote by $R=S^G$ the ring of invariants under this action. This is a $d$-dimensional complete normal domain, and it is called a \emph{quotient singularity}. ADE singularities are $2$-dimensional quotient singularities where $G\subseteq\mathrm{Sl}(2,\kk)$; see Examples \ref{Ex_E6}, and \ref{Ex_A_{n-1}} for some explicit rings of this form.

\par Smith and Van den Bergh showed that quotient singularities have FFRT. More precisely, let $V_0,\dots,V_r$ be a complete set of non-isomorphic irreducible representations of $G$ and let $M_{\alpha}=(S\otimes_{\kk}V_{\alpha})^G$ for $\alpha=0,\dots,r$. In \cite{SmithVDB}, they prove that $R$ has FFRT by the set $\mathcal{N}=\{M_0,\dots,M_r\}$, that is, for every $e\in\mathbb{N}$ the $R$-module $R^{1/p^e}$ is isomorphic to a finite direct sum of elements of $\mathcal{N}$.  
$R$-modules of the form $M=(S\otimes_{\kk}W)^G$, where $W$ is a (not necessarily irreducible) representation of $G$, are called \textit{modules of covariants}.
Direct sums of modules of covariants are still modules of covariants, therefore by Smith and Van den Bergh's result, $R^{1/p^e}$ is a module of covariants as well. We are interested in its decomposition into irreducible modules.

\begin{Remark} The functor $W\mapsto(S\otimes_{\kk}W)^G$, which sends a $\kk$-representation $W$ of $G$ into the corresponding module of covariants, is called \textit{Auslander correspondence}. This gives a one to one correspondence between irreducible $\kk$-representations of $G$ and indecomposable $R$-direct summands of $S$. Moreover, one has $\dim_{\kk}W=\rank_R(S\otimes_{\kk}W)^G$ (see \cite{Auslander} for the original proof in dimension $2$ or \cite[Chapter 5]{LeWi} for a generalization to arbitrary dimension).
\end{Remark}

\begin{Theorem}[Smith-Van den Bergh]\label{theorem-SmithVandenBerg2}
	For any $e\in\mathbb{N}$, let $(S/\m^{[p^e]})^{1/p^e}$ be the Frobenius twist of the representation $S/\m^{[p^e]}$. Then
	\begin{equation*}                                                   
	R^{1/p^e}\cong\left(S\otimes_{\kk}\left((S/\m^{[p^e]})^{1/p^e}\right)\right)^G.     \end{equation*}                                                                    Moreover, if $V_{\alpha}$ is an irreducible $\kk$-representation of $G$ and $M_{\alpha}=(S\otimes_{\kk}V_{\alpha})^G$ is the corresponding module of covariants, then 
	\begin{equation*}
	\mult(M_{\alpha},R^{1/p^e})=\mult(V_{\alpha},(S/\m^{[p^e]})^{1/p^e}).
	\end{equation*}
\end{Theorem}

\begin{Remark}
	Notice that, if $V_0$ is the trivial representation, then $M_0=(S\otimes_{\kk}V_0)^G=R$ and therefore  $\mult(V_{0},(S/\m^{[p^e]})^{1/p^e})=\mult(R,R^{1/p^e})=\frk_R(R^{1/p^e})=FS(e)$ is the F-signature function of $R$.
\end{Remark}

\par Hashimoto and Nakajima \cite{HashimotoNakajima} computed the limits
\[
\lim_{e \to \infty} \frac{\mult(M_{\alpha},R^{1/p^e})}{p^{de}} = \frac{\rank_RM_{\alpha}}{|G|}.
\]
The existence of the previous limits is also a consequence of \cite{SmithVDB, Seibert}, and the value for $\alpha=0$, i.e., the F-signature $\s(R)$, had been previously computed by Watanabe and Yoshida \cite{WatanabeYoshida}.
However, not much is known about the functions $e\mapsto\mult(M_{\alpha},R^{1/p^e})$.
The main result of this section is Theorem \ref{theorem-Fsignaturefunctionquotient}, where we prove that $\mult(M_{\alpha},R^{1/p^e})$ is a quasi-polynomial in $p^e$ and the coefficient of $p^{(d-1)e}$ is always $0$.
Before stating our result, we need the following lemma, which is implicit in \cite{HashimotoNakajima}. Since the methods employed will be useful, we present a complete proof here.
\begin{Lemma}\label{lemma-F-signaturefunction}
Let $G \subseteq {\rm Gl}(d,\kk)$ be as above. For each $g\in G$, we denote by $\lambda_{g,1},\dots,\lambda_{g,d}\in \kk$ its eigenvalues, counted with multiplicity.
	Let $V_{\alpha}$ be an irreducible $\kk$-representation of $G$ with Brauer character $\chi_{V_{\alpha}}$ and associated $R$-module of covariants $M_{\alpha}=(S\otimes_{\kk}V_{\alpha})^G$. The multiplicity of $M_{\alpha}$ into $R^{1/p^e}$ can be expressed as 
	\[
	\mult(M_{\alpha}, R^{1/p^e})=\frac{1}{|G|}\sum_{g\in G}\overline{\chi_{V_{\alpha}}(g)}\sum_{(a_1,\ldots,a_d) \in ([0,p^e-1]\cap\N)^d}\phi\left(((\lambda_{g,1})^{1/p^e})^{a_1}\cdots((\lambda_{g,d})^{1/p^e})^{a_d}\right).
	\]
\end{Lemma}

\begin{proof}
By Theorem \ref{theorem-SmithVandenBerg2}, the multiplicity $\mult(M_{\alpha}, R^{1/p^e})$ is equal to the multiplicity of the representation $V_{\alpha}$ into the Frobenius twist representation $ (S/\m^{[p^e]})^{1/p^e}$.
By Proposition \ref{propertiesofcharacterprop}, this is equal to
\[
\mult(V_{\alpha},(S/\m^{[p^e]})^{1/p^e}) = \frac{1}{|G|}\sum_{g\in G}\overline{\chi_{V_{\alpha}}(g)}\cdot\chi_{(S/\m^{[p^e]})^{1/p^e}}(g).
\]
To compute the previous sum, we fix  an element $g$ of $G$. We may assume without loss of generality that the $\kk$-basis 
 ${x_1,\dots,x_d}$ of the fundamental representation is such that each $x_i$ is an eigenvector of $g$ with eigenvalue $\lambda_{g,i}\in \kk$, that is, $gx_i=\lambda_{g,i}x_i$.
\par Now, observe that $\{x_1^{a_1}\cdots x_d^{a_d}: (a_1,\ldots,a_d) \in ([0,p^e-1]\cap\N)^d\}$ is a $\kk$-basis of $S/\m^{[p^e]}$, 
where each element $x_1^{a_1}\cdots x_d^{a_d}$  is an eigenvector of $g$ with eigenvalue $\lambda_{g,1}^{a_1}\cdots \lambda_{g,d}^{a_d}$.
It follows that $\{(x_1^{1/p^e})^{a_1}\cdots (x_d^{1/p^e})^{a_d} :  (a_1,\ldots,a_d) \in ([0,p^e-1]\cap\N)^d\}$ is a basis of  the Frobenius twist $(S/\m^{[p^e]})^{1/p^e}$ as a $\kk^{1/p^e}$-vector space. Since $\kk$ is perfect, it is a $\kk$-basis as well.
Moreover, each element $(x_1^{1/p^e})^{a_1}\cdots (x_d^{1/p^e})^{a_d}$ of the previous basis is an eigenvector of $g$ with eigenvalue $(\lambda_{g,1}^{1/p^e})^{a_1}\cdots (\lambda_{g,d}^{1/p^e})^{a_d}$. Thus, the character of $(S/\m^{[p^e]})^{1/p^e}$ is given by
\[
\chi_{(S/\m^{[p^e]})^{1/p^e}}(g)=\sum_{(a_1,\ldots,a_d) \in ([0,p^e-1]\cap\N)^d}\phi\left(((\lambda_{g,1})^{1/p^e})^{a_1}\cdots((\lambda_{g,d})^{1/p^e})^{a_d}\right),
\]
and the claim is proved.
\end{proof}

\begin{Definition}
 Let $c\in\{0,\dots,d\}$ and let $g$ be an element of $G \subseteq {\rm Gl}(d,\kk)$. We say that $g$ is a \emph{$c$-pseudoreflection} if it has eigenvalue $1$ with multiplicity $c$, and $d-c$ eigenvalues different from $1$. Equivalently, a $c$-pseudoreflection is an element $g \in {\rm GL}(d,\kk)$ such that $\rank(I_d-g) = d-c$, where $I_d$ is the identity matrix of size $d$. We denote by $G_c$ the subset of $G$ consisting of all $c$-pseudoreflections.
\end{Definition}

Note that, since $G\subseteq \mathrm{Gl}(d,\kk)$ is a finite group whose order is invertible in $\kk$, and $\kk$ is algebraically closed, each element of $G$ is diagonalizable. 
Moreover, observe that we can decompose $G$ as a disjoint union of the sets $G_c$.

\begin{Example}
The only $d$-pseudoreflection corresponds to the identity of the group, and a $(d-1)$-pseudoreflection is just a (standard) pseudoreflection, as in Definition \ref{Def-pseudoreflection}.
\end{Example}

\begin{Remark}
In the literature, $c$-pseudoreflections are sometimes called $(d-c)$-reflections. In particular, (standard) pseudoreflections are sometimes called $1$-reflections, rather than $(d-1)$-pseudoreflections. We decided to adopt this convention in order to facilitate the readability of this article. 
\end{Remark}

\begin{Theorem}\label{theorem-Fsignaturefunctionquotient}
Let $\kk$ be an algebraically closed field of positive characteristic $p$, and let $G$ be a finite small subgroup of $\mathrm{Gl}(d,\kk)$ such that $p\nmid |G|$. Let $S=\kk\llbracket x_1,\dots,x_d\rrbracket$ be a power series ring, and $R=S^G$ be the ring of invariants under this action. Let $V_{\alpha}$ be an irreducible $\kk$-representation of $G$, and $M_{\alpha}=(S\otimes_{\kk}V_{\alpha})^G$ be the corresponding indecomposable module of covariants. Then, the function $e\mapsto \mult(M_{\alpha},R^{1/p^e})$ has the following shape 
\[
\mult(M_{\alpha},R^{1/p^e})=\frac{\rank_RM_{\alpha}}{|G|}p^{de}+\varphi_{d-2}^{(\alpha)}p^{(d-2)e}+\cdots+\varphi_{1}^{(\alpha)}p^e+\varphi_{0}^{(\alpha)},
	\]
where $\varphi_{c}^{(\alpha)} = \varphi_{c}^{(\alpha)}(e)$ are functions that take values in $\Q$, are bounded, and periodic of period at most $|G|-1$. Moreover, if $G$ does not contain any $c$-pseudoreflections for some  $c\in\{0,\dots,d-2\}$, then $\varphi_{c}^{(\alpha)}(e) = 0$. 
\end{Theorem}

\begin{proof}
We fix $e \in \N$. By Lemma \ref{lemma-F-signaturefunction} we can write the multiplicity of $M_{\alpha}$ in $R^{1/p^e}$ as
\[
\mult(M_{\alpha},R^{1/p^e})=\frac{1}{|G|}\sum_{g\in G}\overline{\chi_{\alpha}(g)}\sum_{(a_1,\ldots,a_d) \in ([0,p^e-1]\cap\N)^d}\phi\left(((\lambda_{g,1})^{1/p^e})^{a_1}\cdots((\lambda_{g,d})^{1/p^e})^{a_d}\right),
\]
where $\lambda_{g,1},\dots,\lambda_{g,d}$ are the eigenvalues of the element $g\in G$, and $\chi_{\alpha}$ is the character of $V_{\alpha}$.

\par We write the previous sum as 
\[\begin{split}
&\frac{1}{|G|}\sum_{g\in G}\overline{\chi_{\alpha}(g)}\sum_{(a_1,\ldots,a_d) \in ([0,p^e-1]\cap\N)^d}(\phi((\lambda_{g,1})^{1/p^e}))^{a_1}\cdots(\phi((\lambda_{g,d})^{1/p^e}))^{a_d}\\
=&\frac{1}{|G|}\sum_{g\in G}\overline{\chi_{\alpha}(g)}\sum_{(a_1,\ldots,a_d) \in ([0,p^e-1]\cap\N)^d}(\xi_{g,e,1})^{a_1}\cdots(\xi_{g,e,d})^{a_d},
\end{split}
\]
where $\xi_{g,e,i}=\phi((\lambda_{g,i})^{1/p^e})\in\mathbb{C}$ for all $i=1,\dots,d$.
Notice that since $\phi:\mu_{|G|}(\kk)\rightarrow\mu_{|G|}(\mathbb{C})$ and $(-)^{1/p^e}:\mu_{|G|}(\kk)\rightarrow\mu_{|G|}(\kk)$ are group isomorphisms, the order of $\xi_{g,e,i}$ as root of unity in $\mathbb{C}$ is the same as the order of $\lambda_{g,i}$ in $\kk$.
\par Now, rewrite the sum as 
\begin{equation}\label{eq_F-signature1}\begin{split}
&\frac{1}{|G|}\sum_{g\in G}\overline{\chi_{\alpha}(g)}\sum_{a_1=0}^{p^e-1}(\xi_{g,e,1})^{a_1}\cdots\sum_{a_d=0}^{p^e-1}(\xi_{g,e,d})^{a_d}\\
=&\frac{1}{|G|}\sum_{g\in G}\overline{\chi_{\alpha}(g)}\prod_{i=1}^d\sum_{a_i=0}^{p^e-1}(\xi_{g,e,i})^{a_i}\\
=&\frac{1}{|G|}\sum_{c=0}^d\sum_{g\in G_c}\overline{\chi_{\alpha}(g)}\prod_{i=1}^d\sum_{a_i=0}^{p^e-1}(\xi_{g,e,i})^{a_i}.
\end{split}
\end{equation}
The last equality follows from the disjoint decomposition $G=\bigsqcup_{c=0}^dG_c$, where $G_c$ is the set of $c$-pseudoreflections.

\par We analyze the last formula more closely. First, observe that each sum of the form $\sum_{a_i=0}^{p^e-1}(\xi_{g,e,i})^{a_i}$ is equal to $p^e$ if $\lambda_{g,i}=1$. In fact, in this case, $\xi_{g,e,i}=1$ for all $e$. On the other hand, if $\lambda_{g,i} \ne 1$, then the function $e \mapsto \left|\sum_{a_i=0}^{p^e-1}(\xi_{g,e,i})^{a_i}\right|$ is bounded by a constant. In fact, $\lambda_{g,i} \ne 1$ if and only if $\xi_{g,e,i} \ne 1$ for all $e$, by Remark \ref{Remark_order_roots}. 
We fix $n=|G|$, and write $p^e=kn+r_e$, with $0 < r_e < n$. Since $\xi_{g,e,i}\ne 1$, we have $\sum_{a_i=jn}^{(j+1)n-1} (\xi_{g,e,i})^{a_i} =  0$ for all $j=0,\ldots, k-1$, and thus $\sum_{a_i=0}^{p^e-1} (\xi_{g,e,i})^{a_i} = \sum_{a_i=kn}^{p^e-1} (\xi_{g,e,i})^{a_i}$.

\par Now, fix $c\in\{0,\dots,d\}$. Following the previous argument, for each, $g\in G_c$ and all $e$ we have exactly $c$ eigenvalues in the set $\{\xi_{g,e,1},\dots,\xi_{g,e,d}\}$ which are equal to $1$. Therefore,
\[
\prod_{i=1}^d\sum_{a_i=0}^{p^e-1}(\xi_{g,e,i})^{a_i}=\eta_{g,c}p^{ce}
\]
for some function $\eta_{g,c} = \eta_{g,c}(e)$ that, for all $e \in \N$, satisfies $|\eta_{g,c}(e)| < C$ for some $C>0$ independent of $e$. 
Taking the sum over all $g\in G_c$, we obtain 
\[
\frac{1}{|G|}\sum_{g\in G_c}\overline{\chi_{\alpha}(g)}\prod_{i=1}^d\sum_{a_i=0}^{p^e-1}(\xi_{g,e,i})^{a_i}=\varphi_{c}^{(\alpha)}p^{ce},
\]
where $\varphi_{c}^{(\alpha)} = \varphi_{c}^{(\alpha)}(e) = \frac{1}{|G|}\sum_{g\in G_c}\overline{\chi_{\alpha}(g)} \eta_{g,c}(e)$. Note that $\left|\varphi_c^{(\alpha)}(e)\right|$ can be also bounded by a constant independent of $e$, because $|G_c|$ and $\left|\overline{\chi_{\alpha}(g)}\right|$ are independent of  $e$. 
Inserting the last formula in \eqref{eq_F-signature1}, we get
\[\mult(M_{\alpha},R^{1/p^e})=\sum_{c=0}^d\varphi_{c}^{(\alpha)}p^{ce}.
\]
This shows that $\mult(M_\alpha,R^{1/p^e})$ is a quasi-polynomial; the fact that $\mult(M_\alpha,R^{1/p^e}) \in \N$ for all $e \in \N$ now gives that that the functions $\varphi_c^{(\alpha)}$ take values in $\Q$. In addition, it is clear from the description of $\varphi_c^{(\alpha)}$ that $G_c=\emptyset$ implies $\varphi_{c}^{(\alpha)}=0$. Therefore, since $G$ does not contain any $(d-1)$-pseudoreflections, we have $G_{d-1}=\emptyset$, and consequently  $\varphi_{d-1}^{(\alpha)}(e)=0$ for all $e \in \N$. Furthermore, $\varphi_{d}^{(\alpha)}=\frac{\rank_RM_{\alpha}}{|G|}$ follows from the fact that $G_d=\{\mathrm{Id}_G\}$, and $\overline{\chi_{\alpha}(\mathrm{Id}_G)}=\dim_{\kk}V_{\alpha}=\rank_RM_{\alpha}$. 
 \par It is left to show that the functions $e \mapsto \varphi_{c}^{(\alpha)}(e)$ are periodic. It is enough to show that each function $e \mapsto \sum_{a_i=0}^{p^e-1} (\xi_{g,e,i})^{a_i}$ is periodic, for all $g$ and $i$ such that $\lambda_{g,i} \ne 1$. Since $p \nmid n$, where $n=|G|$, we can find $e'$ such that $p^{e'} \equiv 1$ modulo $n$. Note that we can choose $e'$ to be the order of $p$ in the group of units of $\Z/(n)$; in particular, we can assume that $e' \leq |G|-1$. Observe that $\lambda_{g,i}^{p^{e'}} = \lambda_{g,i}$, because $\lambda_{g,i}^n=1$. Since $(\phi^{-1}(\xi_{g,1,i}))^p=\lambda_{g,i}$, we get $(\phi^{-1}(\xi_{g,1,i}))^{pp^{e'}} = \lambda_{g,i} =  (\phi^{-1}(\xi_{g,1,i}))^p$, and it follows that $\xi_{g,1,i} = \xi_{g,e'+1,i}$. Finally, since this is true for all $g$ and $i$ such that $\lambda_{g,i} \ne 1$, we have that $\varphi_{c}^{(\alpha)}(e) = \varphi_{c}^{(\alpha)}(e+e')$ for all $e \in \N$.
\end{proof}

We postpone to Section \ref{section:examples} the presentation of some examples, which show how Theorem \ref{theorem-Fsignaturefunctionquotient} can be used to compute the F-signature function of specific quotient singularities (see e.g. Example~\ref{Ex_E6} and Example~\ref{Ex_3-VeroneseD6}).
\par We have shown in Theorem \ref{theorem-Fsignaturefunctionquotient} that $G_c=\emptyset$ implies that $\varphi_c^{(\alpha)}=0$ for all $\alpha$. We can prove a converse statement, provided $\alpha=0$. In other words, the vanishing of the function $\varphi_c^{(0)}$ is equivalent to the absence of $c$-pseudoreflections.
In order to simplify the notation, in the sequel when no confusion may arise, we will simply denote the function $\varphi_c^{(0)}$ by $\varphi_c$.

\begin{Proposition}\label{prop-Fsignaturefuncionquotient}
With the notations of Theorem \ref{theorem-Fsignaturefunctionquotient}, for any $c\in\{0,\dots,d-2\}$, we have $\varphi_{c}(e) = 0$ for all $e\in\mathbb{N}$ if and only if $G$ does not contain $c$-pseudoreflections.
\end{Proposition}
\begin{proof}
\par The \textit{if} part of the statement has been proved in Theorem \ref{theorem-Fsignaturefunctionquotient}, so it remains to prove the \textit{only if} part.
 For this, fix $e'$ such that $p^{e'}\equiv1$ modulo $|G|$. 
 For $g\in G_c$  we denote by $\lambda_{g,1},\dots,\lambda_{g,d}$ its eigenvalues, and we set $\xi_{g,e',i}=\phi((\lambda_{g,i})^{1/p^{e'}})\in\mathbb{C}$ as in the proof of Theorem \ref{theorem-Fsignaturefunctionquotient}.  Since $g$ is a $c$-pseudoreflection, there will be exactly $c$ values from the set $\{\xi_{g,e',1},\dots,\xi_{g,e',d}\}$ that are equal to $1$. Without loss of generality, we may assume that $\xi_{g,e',1}=\dots=\xi_{g,e',c}=1$. Using the formula for $\varphi_c$ obtained inside the proof of Theorem \ref{theorem-Fsignaturefunctionquotient}, the assumption that $\varphi_c(e')=0$ gives
\begin{equation*}
\begin{split}
0&=\frac{1}{|G|}\sum_{g\in G_c}\prod_{i=1}^d\sum_{a_i=0}^{p^{e'}-1}(\xi_{g,p^{e'},i})^{a_i}\\
&=\frac{1}{|G|}\sum_{g\in G_c}\prod_{i=c+1}^d\sum_{a_i=0}^{p^{e'}-1}(\xi_{g,p^{e'},i})^{a_i}(p^{e'})^c\\
&=\frac{1}{|G|}\sum_{g\in G_c}\left(\prod_{i=c+1}^d1\right)(p^{e'})^c\\
&=\frac{1}{|G|}|G_c|(p^{e'})^c,
\end{split}
\end{equation*}
which implies $|G_c|=0$.
In the previous chain of equalities from the second to the third line we used the fact that
\[\sum_{a_i=0}^{p^{e'}-1}(\xi_{g,e',i})^{a_i}=\xi_{g,e',i}^{0}=1,
\] 
which is true because of our choice of $p^{e'} \equiv 1$ modulo $|G|$.
\end{proof}

The following Corollary is a direct consequence of the proof of Proposition \ref{prop-Fsignaturefuncionquotient}.
\begin{Corollary}\label{corollary-fsignaturecoprime}
For any $e\in\mathbb{N}$ such that $p^e\equiv 1$ modulo $|G|$, we have
	\begin{equation*}
FS(e)=\frac{1}{|G|}p^{de}+\frac{|G_ {d-2}|}{|G|}p^{(d-2)e}+\cdots+\frac{|G_1|}{|G|}p^e+\frac{|G_0|}{|G|}.
	\end{equation*}
In particular, if $p\equiv 1$ modulo $|G|$, then this is true for all $e\in\mathbb{N}$ so the F-signature function of $R$ is a polynomial in $p^e$ with constant coefficients.
\end{Corollary}

\begin{Remark} \label{rem_graded}
	We state the results of this section in the complete local case; however, analogous versions are true in the graded setting.
	More precisely, let $\kk$ be an algebraically closed field of characteristic $p>0$, let $S=\kk[x_1,\dots,x_d]$ with $\deg x_i=1$, and let $G\subseteq\mathrm{Gl}(d,\kk)$ be a finite small group with $p\nmid |G|$. We consider the corresponding invariant ring $R=S^G$, which is $\mathbb{N}$-graded. 
	The multiplicity functions $\mult(M_{\alpha},R^{1/p^e})$ are defined similarly to the local case (see \cite[Section 3.1]{SmithVDB} for more details).
	The Auslander correspondence between irreducible $\kk$-representations of $G$ and graded indecomposable $R$-direct summands of $S$ is true also in this setting (see \cite[Section 4]{IyamaTakahashi} for a proof) and a graded version of Theorem \ref{theorem-SmithVandenBerg2} has been proved by Hashimoto and Nakajima \cite[Proposition 2.2]{HashimotoNakajima}.
	Therefore,  Theorem \ref{theorem-Fsignaturefunctionquotient}, Proposition \ref{prop-Fsignaturefuncionquotient}, and Corollary \ref{corollary-fsignaturecoprime} hold in this setting as well with analogous proofs.
\end{Remark}

\section{F-signature function of cyclic quotient singularities}\label{section:cyclicquotient}
Let $S=\kk\ps{x_1,\ldots,x_d}$, where $\kk$ is an algebraically closed field of characteristic $p>0$. Let $G \subseteq \mathrm{Gl}(d,\kk)$ be a finite small subgroup of order $n$, with $p \nmid n$. Throughout this section, we assume that $G$ is cyclic.
In particular, we may assume that $G$ is generated by an element $g=\mathrm{diag}(\lambda^{t_1},\dots,\lambda^{t_d})$, where $\lambda\in \kk$ is a primitive $n$-th root of unity and $t_1,\dots,t_d$ are non-negative integers.
It is harmless to assume $\gcd(t_1,\ldots,t_d,n)=1$. Moreover, since $G$ is small, we must have $\gcd(t_{j_1},\ldots,t_{j_{d-1}},n)=1$ for all subsets $\{j_1,\ldots,j_{d-1}\} \subseteq [d]$ of cardinality $d-1$. The ring $R=S^G$ of invariants with respect to the action of $G$ is called a cyclic quotient singularity, which we will denote by $\frac{1}{n}(t_1,t_2,\ldots,t_d)$. In this setup, we can apply Theorem \ref{theorem-Fsignaturefunctionquotient} to describe the functions $e \to \mult(M_\alpha,R^{1/p^e})$. However, given the special structure of the group $G$, we can say more about these functions.

\begin{Remark} \label{Remark_irreps_cyclic} When $G$ is a cyclic small group of order $n$, there are precisely $n$ irreducible $\kk$-representations $V_{0},\dots,V_{n-1}$ of $G$, and they all have rank $1$. Furthermore, for $\alpha\in\{0,\dots,n-1\}$, the Brauer character $\chi_{V_\alpha}$ will be of the form $\xi^j$ for some $0 \leq j \leq n-1$ and some primitive $n$-th root of unity $\xi\in\mathbb{C}$. We will then assume, without loss of generality, that the irreducible $\kk$-representations are such that $\chi_{V_\alpha} = \xi^{\alpha}$, for all $\alpha$.
\end{Remark}

In what follows, we denote by $\mathcal{P}=[0,1]^d$ the unitary cube of side $1$ inside $\R^d$ and, for each $\alpha \in \{0,\ldots,n-1\}$, we let $\mathcal{A}^{(\alpha)}$ be the lattice 
	\begin{equation}\label{eq:latticeA}
	\ds \mathcal{A}^{(\alpha)}=\{(a_1,\ldots,a_{d}) \in  \Z^d : t_1a_1 + t_2a_2 + \ldots + t_{d}a_{d} \equiv \alpha \mod n\}.
	\end{equation}
We start by relating the functions $e \mapsto \mult(M_\alpha, R^{1/p^e})$ to the number of lattice points inside multiples of the cube $\mathcal{P}$. 
\begin{Proposition}\label{prop-Fsignatureofcyclic}
Let $R$ be a $\frac{1}{n}(t_1,t_2,\dots,t_d)$-cyclic singularity over an algebraically closed field, and let $e\in\mathbb{N}$. Then
	\[
	\ds \mult(M_\alpha,R^{1/p^e})= |(p^e-1)\mathcal{P}\cap \mathcal{A}^{(\alpha)}|,
	\]
\end{Proposition}

\begin{proof}
Since $G$ is cyclic, its elements can be written as $g^j$ for $j=0,\dots,n-1$. In particular, observe that the eigenvalues of $g^j$ are $\lambda^{jt_1}$, $\lambda^{jt_2},\dots,\lambda^{jt_d}$.
Let $\xi_e = \phi(\lambda^{1/p^e})$ be the image in $\CC$ of the unique $p^e$-root of $\lambda$ in $\kk$. Notice that $\xi_e$ is a primitive complex $n$-th root of unity, so by Remark \ref{Remark_irreps_cyclic} we may assume  that $\chi_{V_\alpha}=\xi_e^\alpha$.
Observe that $\overline{\chi_{V_\alpha}} = \overline{\xi_e^\alpha} = \xi_e^{-\alpha}$. Then from Lemma \ref{lemma-F-signaturefunction} we obtain 
\begin{align*}
\mult(M_\alpha,R^{1/p^e}) & = \frac{1}{n} \sum_{j=0}^{n-1} \xi_e^{-j\alpha} \sum_{\tiny{(a_1,\ldots,a_d) \in ([0,p^e-1]\cap \N)^d}} \xi_e^{j(t_1a_1+t_2a_2 + \ldots + t_{d}a_{d})} \\
&  = \frac{1}{n} \sum_{j=0}^{n-1} \sum_{\tiny{(a_1,\ldots,a_d) \in ([0,p^e-1]\cap \N)^d}} \xi_e^{j(t_1a_1+t_2a_2 + \ldots + t_{d}a_{d}-\alpha)}.
\end{align*}
Since $\sum_{j=0}^{n-1} \xi_e^{ij} = 0$ for all $i \not\equiv 0$ modulo $n$, the only contribution to the sum above is for $(a_1,\ldots,a_d)$ such that $t_1a_1+t_2a_2 + \ldots + t_{d}a_{d} \equiv \alpha$ modulo $n$, in which case $\xi_e^{t_1a_1+t_2a_2 + \ldots + t_{d}a_{d}-\alpha} = 1$. Therefore
\begin{align*}
\mult(M_\alpha,R^{1/p^e}) & = \frac{1}{n} \sum_{j=0}^{n-1} \sum_{\tiny{\begin{array}{c} t_1a_1+t_2a_2 + \ldots + t_{d}a_{d} \equiv \alpha \mod n \\ (a_1,\ldots,a_d) \in ([0,p^e-1]\cap \N)^d \end{array}}} 1 \\
& = |\{(a_1,\ldots,a_{d}) \in ([0,p^e-1] \cap \N)^d : t_1a_1 + t_2a_2 + \ldots + t_{d}a_{d} \equiv \alpha \mod n\}|\\
& = |(p^e-1)\mathcal{P} \cap \mathcal{A}^{(\alpha)}|.
\end{align*}
\end{proof}

Proposition \ref{prop-Fsignatureofcyclic} exhibits a connection between the F-signature function of cyclic quotient singularities and Erhart functions of rational polytopes. This is not surprising: in fact, cyclic quotient singularities are toric, and Von Korff proved that the F-signature function of toric rings is an Erhart function \cite{VonKorff} (see also \cite{Bruns} for related results). However, while in Von Korff's approach the lattice is $\mathbb{Z}$ and the polytope is not a cube, in Proposition \ref{prop-Fsignatureofcyclic} the lattice is more complicated, but the polytope is a cube. The advantage of our method is that it allows to compute the coefficients of the quasi-polynomial $\mult(M_\alpha,R^{1/p^e})$ more explicitly, and to relate them to properties of the group $G$ (see Theorem \ref{theorem-Fsignaturecyclic}).

\subsection{Congruences and partitions} \label{Subsection_congruence} 
In this subsection we recall some well-known facts about congruences modulo an integer.
The results and the methods of this subsection are general in nature and independent of the cyclic quotient singularities setting.
However, the notation we introduce and lemmas we prove here will be used in the rest of Section \ref{section:cyclicquotient}.

The following Lemma about number of solutions of certain congruence relations is a well-known classical result, therefore we omit a proof.
\begin{Lemma}\label{idolo} Let $t_1,\ldots,t_i,n,b$ be non-negative integers, with $n\ne 0$, and $g=\gcd(t_1,\ldots,t_i,n)$ that divides $b$. The congruence $t_1x_1 + \ldots + t_ix_i \equiv b$ modulo $n$ has $g \cdot n^{i-1}$ incongruent solutions $(x_1,\ldots,x_i) \in \Z/(n)^{\oplus i}$.
\end{Lemma}

We now introduce some notation that will largely be used in the rest of this section.
\begin{Notation} \label{notation_Gamma}
Fix positive integers $d,n$ and $p>1$, with $\gcd(p,n)=1$.  Fix a natural number $e$, and  write $p^e=nk+r_e$, with $0 <  r_e < n$. For every $0 \leq i \leq d$ we let $\Gamma_i =\{ J \subseteq [d]  : |J|=i\}$. For $J \in \Gamma_i$, we let $C_J = \prod_{j=1}^d \left([0,b_j] \cap \N\right) \subseteq \N^d$, with 
\[
\ds b_j= \left\{\begin{array}{ll} n-1 & \mbox{ if } j \in J \\ \\
 r_e-1 & \mbox{ if } j \notin J
 \end{array}
 \right.
\]
For example, we have $C_{[d]} = ([0,n-1] \cap \N)^d$, and $C_{\emptyset}= ([0,r_e-1]\cap \N)^d$.  

Now let $1 \leq i \leq d$. For $J = \{j_1,\ldots,j_i\} \subseteq \Gamma_i$, with $j_1 < j_2 < \ldots, j_i$, we let $\sigma_J :J \to [i]$ be the map defined as $\sigma_J(j_\ell) = \ell$ for all $1 \leq \ell \leq i$. For $\underline{s}=(s_1,\ldots,s_i) \in \N^i$ we define a vector $v_{J,\underline{s}} = ((v_{J,\underline{s}})_1,\ldots,(v_{J,\underline{s}})_d) \in \N^d$ in the following way:
\[
\ds (v_{J,\underline{s}})_j= \left\{ \begin{array}{ll} s_{\sigma_J(j)} & \mbox{ if } j \in J \\ \\
k & \mbox{ if } j\notin J
\end{array}
\right.
\]  
Finally, for convenience, we set $([0,k-1]\cap\N)^0=\{\star\}$, and $v_{\emptyset,\star} = (k,\ldots,k)$.
\end{Notation}
Given a set $C \subseteq \N^d$ and a $d$-uple $(a_1,\ldots,a_d)$, we denote by $C+(a_1,\ldots,a_d)$ the Minkowski sum $\{(c_1+a_1,\ldots,c_d+a_d) : (c_1,\ldots,c_d)\in C\}$. We will call it the shift of the set $C$ by $(a_1,\ldots,a_d)$. With the notation we have introduced, we can partition the set $([0,p^e-1]\cap \N)^d$ into shifts of sets of the form $C_J$, for $J \subseteq [d]$.
\begin{Lemma} \label{partition}
We have the following partition:
\begin{align*}
\ds \left([0,p^e-1] \cap \N\right)^d & = \bigsqcup_{i=0}^d \bigsqcup_{J \in \Gamma_i} \left(\bigsqcup_{\tiny{ \underline{s} \in ([0,k-1] \cap \N)^i}} (C_J + nv_{J,\underline{s}}) \right).
\end{align*}
\end{Lemma}
Note that, on the right-hand side of the equation, the sets $C_J$ and the vectors $v_{J,\underline{s}}$ depend on $e$. At this stage, we have decided to keep the dependence of these objects on $e$ implicit, since we believe this should not be source of confusion, while adding it to the notation would only make the statement harder to read. Before starting the proof, to better illustrate our notation and the statement of the Lemma, we display the partition when $d=p=2$ and $e=n=3$, in which case we have $k=r_e=2$.
\begin{figure}[h]
\begin{tikzpicture}
  [decoration={markings,mark=at position 1 with {\arrow{stealth}}},
   blob/.style={circle,fill,minimum width=4pt,inner sep=0pt},
   flow/.style={postaction={decorate}}
  ]
  \node[blob, label=left:0] (O) at (-1, 0) {};
  \node (Y) at (-1, 7.5) {};
  \node (X) at (6.5,0) {};
  \draw[flow] (O)  -- (Y);
  \draw[flow] (O)  -- (X);
  \node[blob, label=left:1] (y1) at (-1,1) {};
  \node[blob, label=left:2] (y2) at (-1,2) {};
  \node[blob, label=left:3] (y3) at (-1,3) {};
  \node[blob, label=left:4] (y4) at (-1,4) {};
  \node[blob, label=left:5] (y5) at (-1,5) {};
  \node[blob, label=left:6] (y6) at (-1,6) {};
  \node[blob, label=left:7] (y7) at (-1,7) {};
  \node[blob, label=below:1] (x1) at (0,0) {};
  \node[blob, label=below:2] (x2) at (1,0) {};
  \node[blob, label=below:3] (x3) at (2,0) {};
  \node[blob, label=below:4] (x4) at (3,0) {};
  \node[blob, label=below:5] (x5) at (4,0) {};
  \node[blob, label=below:6] (x6) at (5,0) {};
  \node[blob, label=below:7] (x7) at (6,0) {};
  \node[blob] () at (0,1) {};
  \node[blob] () at (1,1) {};
  \node[blob] () at (2,1) {};
  \node[blob] () at (3,1) {};
  \node[blob] () at (4,1) {};
  \node[blob] () at (5,1) {};
  \node[blob] () at (6,1) {};
  \node[blob] () at (0,2) {};
  \node[blob] () at (1,2) {};
  \node[blob] () at (2,2) {};
  \node[blob] () at (3,2) {};
  \node[blob] () at (4,2) {};
  \node[blob] () at (5,2) {};
  \node[blob] () at (6,2) {};
  \node[blob] () at (0,3) {};
  \node[blob] () at (1,3) {};
  \node[blob] () at (2,3) {};
  \node[blob] () at (3,3) {};
  \node[blob] () at (4,3) {};
  \node[blob] () at (5,3) {};
  \node[blob] () at (6,3) {};
  \node[blob] () at (0,4) {};
  \node[blob] () at (1,4) {};
  \node[blob] () at (2,4) {};
  \node[blob] () at (3,4) {};
  \node[blob] () at (4,4) {};
  \node[blob] () at (5,4) {};
  \node[blob] () at (6,4) {};
  \node[blob] () at (0,5) {};
  \node[blob] () at (1,5) {};
  \node[blob] () at (2,5) {};
  \node[blob] () at (3,5) {};
  \node[blob] () at (4,5) {};
  \node[blob] () at (5,5) {};
  \node[blob] () at (6,5) {};
  \node[blob] () at (0,6) {};
  \node[blob] () at (1,6) {};
  \node[blob] () at (2,6) {};
  \node[blob] () at (3,6) {};
  \node[blob] () at (4,6) {};
  \node[blob] () at (5,6) {};
  \node[blob] () at (6,6) {};
  \node[blob] () at (0,7) {};
  \node[blob] () at (1,7) {};
  \node[blob] () at (2,7) {};
  \node[blob] () at (3,7) {};
  \node[blob] () at (4,7) {};
  \node[blob] () at (5,7) {};
  \node[blob] () at (6,7) {};
  \draw[dashed] (-1,0) rectangle (1,2);
  \node at (-0.3,1) (a1) {};
  \node at (-3,1.5) () {$C_{[2]}$};
  \node at (-3,1) (a) {$+$};
  \node at (-3,0.5) () {$3v_{[2],(0,0)}$};
  \draw[->] (a) to [out=20,in=160] (a1);
  \draw[dashed] (-1,3) rectangle (1,5);
  \node at (-0.3,4) (b1) {};
  \node at (-3,4.5) () {$C_{[2]}$};
  \node at (-3,4) (b) {$+$};
   \node at (-3,3.5) () {$3v_{[2],(0,1)}$};
   \draw[draw,->] (b) to [out=20,in=160] (b1);
   \draw[dashed] (2,0) rectangle (4,2);
   \node at (3.5,0.5) (c1) {};
  \node at (3,-1) (c) {$C_{[2]}$};
  \node at (3,-1.5) () {$+$};
  \node at (3,-2) () {$3v_{[2],(1,0)}$};
  \draw[->] (c) to [out=70,in=290] (c1);
  \draw[dashed] (2,3) rectangle (4,5);
  \node at (3.5,4.5) (d1) {};
  \node at (5,9) () {$C_{[2]}$};
  \node at (5,8.5) () {$+$};
  \node at (5,8) (d) {$3v_{[2],(1,1)}$};
  \draw[->] (d) to [out=260,in=70] (d1);
  \draw[dashed] (5,0) rectangle (6,2);
  \node at (5.5,1) (d) {};
  \node at (7,1.5) () {$C_{\{2\}}$};
  \node at (7,1) (e) {$+$};
  \node at (7,0.5) () {$3v_{\{2\},(0)}$};
  \draw[draw,->] (e) to [out=160,in=20] (d);
  \draw[dashed] (5,3) rectangle (6,5);
  \node at (5.5,4) (f) {};
  \node at (7,4.5) () {$C_{\{2\}}$};
  \node at (7,4) (g) {$+$};
  \node at (7,3.5) () {$3v_{\{2\},(1)}$};
  \draw[draw,->] (g) to [out=160,in=20] (f);
  \draw[dashed] (5,6) rectangle (6,7);
  \node at (5.5,6.5) (h) {};
  \node at (7,7) () {$C_{\emptyset}$};
  \node at (7,6.5) (i) {$+$};
  \node at (7,6) () {$3v_{\emptyset,\star}$};
  \draw[draw,->] (i) to [out=160,in=20] (h);
  \draw[dashed] (-1,6) rectangle (1,7);
  \node at (-0.5,6.5) (j) {};
  \node at (-1,9) () {$C_{\{1\}}$};
  \node at (-1,8.5) () {$+$};
  \node at (-1,8) (k) {$3v_{\{1\},(0)}$};
  \draw[draw,->] (k) to [out=290,in=70] (j);
  \draw[dashed] (2,6) rectangle (4,7);
  \node at (2.5,6.5) (l) {};
  \node at (2,9) () {$C_{\{1\}}$};
  \node at (2,8.5) () {$+$};
  \node at (2,8) (m) {$3v_{\{1\},(1)}$};
  \draw[draw,->] (m) to [out=290,in=70] (l);
\end{tikzpicture}
\caption{Case $d=p=2$, $e=n=3$}
\label{figure}
\end{figure}
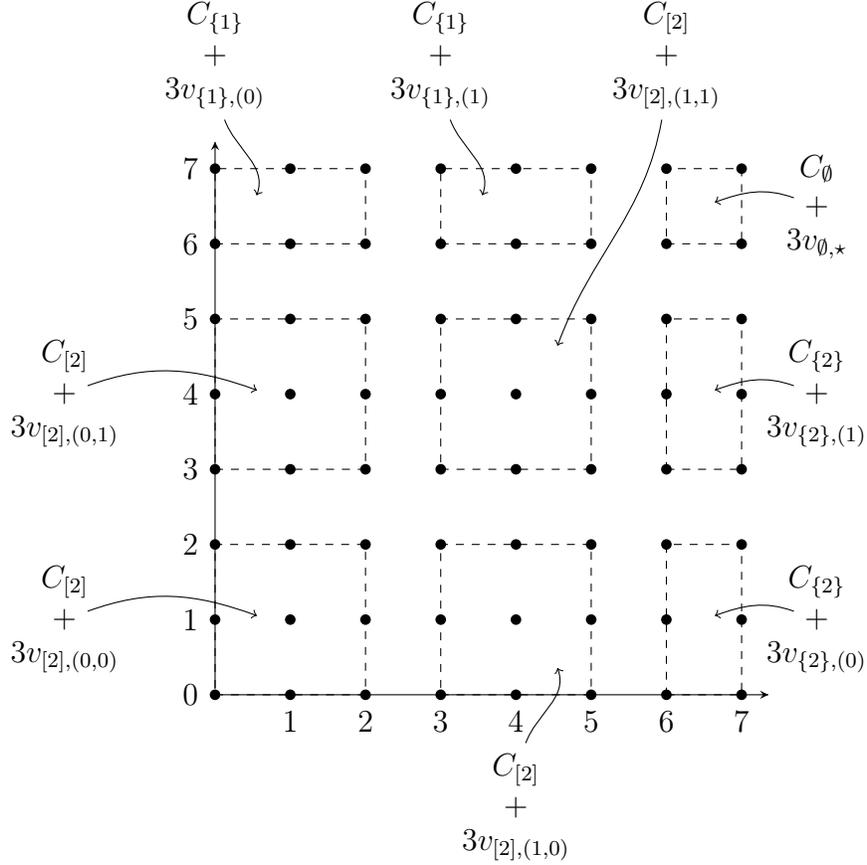

\begin{proof}[Proof of Lemma \ref{partition}]
For $\underline{s} \in ([0,k-1]\cap \N)^i$, each $C_J + nv_{J,\underline{s}}$ is contained in $([0,p^e-1]\cap \N)^d$. Therefore, the union of such sets is contained in $([0,p^e-1]\cap\N)^d$ as well. To see the other containment, let $(a_1,\ldots,a_d) \in ([0,p^e-1]\cap \N)^d$. Let $J = \{j_1,\ldots,j_i\}$, with $j_1< \ldots < j_i$, be the set of $j \in [d]$ such that $a_j<kn$. If $J=\emptyset$, then $(a_1,\ldots,a_d) \in ([kn,p^e-1]\cap \N)^d = C_\emptyset + v_{\emptyset,\star}$. If $J \ne \emptyset$, for each $j_\ell \in J$, write $a_{j_\ell} = s_{j_\ell}n + r_{j_\ell}$, with $0 \leq r_{j_\ell} < n$, and set $\underline{s} = (s_{j_1},\ldots,s_{j_i})$. With these choices, one can check that $(a_1,\ldots,a_d) \in C_J+nv_{J,\underline{s}}$. It is also straightforward, and we leave it to the reader, to check that all the sets $C_J + nv_{J,\underline{s}}$ are disjoint.
\end{proof}

We need some additional notation.
\begin{Notation} \label{notation_psi}
Fix positive integers $d,n$ and $p$, with $\gcd(p,n)=1$, and non-negative integers $t_1,\dots,t_d$. Let $0\leq i< d$ be an integer, and  $J \in \Gamma_i$. Write $J=\{j_1,\ldots,j_i\}$, and let $g_J = \gcd(t_{j_1},\ldots,t_{j_i},n)$, where for convenience we set $g_\emptyset = n$. We let $[d] \smallsetminus J = \{j_{h_1},\ldots,j_{h_{d-i}}\}$. 

Given a positive integer $e$, write $p^e=kn+r_e$, with $0 < r_e<n$. For $\alpha \in \{0,\ldots,n-1\}$, let
\[
\ds \mathcal{B}_J^{(\alpha)}(e)= \left\{(a_{1},\ldots,a_{d-i}) \in ([0,r_e-1] \cap \N)^{d-i}  : \sum_{\ell=1}^{d-i} a_\ell t_{h_\ell} \equiv \alpha\mod g_J\right\}.
\]
Finally, for all $e \in \N$, we define $\theta_J^{(\alpha)}(e)$ to be the cardinality of the set $\mathcal{B}_J^{(\alpha)}(e)$. 
\par In other words, $\theta_J^{(\alpha)}(e)$ counts the number of incongruent $(d-i)$-uples $(\overline{a_1},\ldots,\overline{a_{d-i}})$ in $\mathbb{Z}/(r_e)^{\oplus (d-i)}$ such that their lifts $(a_1,\ldots,a_{d-i})$ to $\Z$, with $0\leq a_\ell \leq r_e-1$, satisfies $\sum_{\ell =1}^{d-i} a_\ell t_{h_\ell} +ag_J = \alpha$ for some $a \in \Z$. For convenience, we set $\theta_{[d]}^{(\alpha)}(e) = 1$ for all $\alpha=0,\ldots, n-1$, and $e \in \N$. For $i\in\{0,\ldots,d\}$ and $\alpha \in\{ 0,\ldots,n-1\}$, consider the following functions
\[
e \in \N \mapsto \psi_i^{(\alpha)}(e) = \sum_{J \in \Gamma_i} g_J\theta_J^{(\alpha)}(e).
\] 
When no confusion may arise, we will simply denote $\theta_J^{(0)}$ and $\psi_J^{(0)}$ by $\theta_J$ and $\psi_J$. 
\end{Notation}
The functions $\theta_J^{(\alpha)}$ count the number of solutions of certain diophantine equations in linearly bounded regions. 
This problem has been studied in the context of \emph{integer linear programming}.
The interested reader may consult \cite{IntegerProgrammingBook} for an introduction to this research area.
A recursive formula for  $\theta_J^{(\alpha)}$ can be also deduced from \cite{Faaland}.

\begin{Remark} \label{rem_psi} 
For all $J \in \Gamma_i$, $\alpha\in\{0,\ldots,n-1\}$, and $e \in \N$, we have bounds $0 \leq \theta_J^{(\alpha)}(e) \leq r_e^{d-i}$. The upper bound is clear from the range where we $(a_1,\ldots,a_{d-i})$ varies. When $\alpha=0$, the lower bound can be improved to $1 \leq \theta_J^{(0)}(e)$ for all $J$ and all $e$, since the $(d-i)$-uple $(0,\ldots,0)$ always belongs to $\mathcal{B}_J^{(0)}(e)$. Note that the upper bound $\theta_J^{(\alpha)}(e) = r_e^{d-i}$ is always achieved, independently of $\alpha$, when $g_J = 1$. Moreover, if $r_e=1$, then for all $J$ and all $e$ we have $\theta^{(\alpha)}_J(e) = 0$ when $\alpha \ne 0$, while $\theta^{(0)}_J(e) = 1$.
\end{Remark}
\begin{Proposition} \label{counting C_J} We adopt the notation introduced in  \ref{notation_Gamma} and \ref{notation_psi}. For $\alpha\in\{0,\dots,n-1\}$, we consider the lattice $ \mathcal{A}^{(\alpha)}$	defined in \eqref{eq:latticeA}.
Let $0 \leq i \leq d$ be an integer, and $J \in \Gamma_i$. Write $J=\{j_1,\ldots,j_i\}$  and let $g_J = {\rm gcd}(t_{j_1},\ldots,t_{j_i},n)$. For $e \in \N$, write $p^e=kn+r_e$, with $0 < r_e < n$. Then
\[
\ds \left|\left(C_J + nv_{J,\underline{s}}\right) \cap \mathcal{A}^{(\alpha)} \right| = |C_J \cap \mathcal{A}^{(\alpha)}| = \theta_J^{(\alpha)}(e)g_Jn^{i-1}.
\]
for all $\underline{s} \in ([0,k-1] \cap \N)^i$.
\end{Proposition}

\begin{proof}
To prove the first equality, let $\underline{s}\in ([0,k-1] \cap \N)^i$ be arbitrary. Then
\begin{align*}
(a_1,\ldots,a_d) \in \left(C_J + nv_{J,\underline{s}}\right) \cap \mathcal{A}^{(\alpha)}  & \Longleftrightarrow \left\{ \begin{array}{ll} (a_1,\ldots,a_d) -nv_{J,\underline{s}} \in C_J \\ \\ t_1a_1+t_2 a_2+\ldots+t_d a_d \equiv \alpha  \mod n \end{array}  \right. \\ \\
& \Longleftrightarrow \left\{ \begin{array}{ll}  
(a_1,\ldots,a_d) -nv_{J,\underline{s}} \in C_J \\ \\ 
t_1(a_1-n(v_{J,\underline{s}})_1)+\ldots+t_d (a_d -n(v_{J,\underline{s}})_d) \equiv \alpha  \mod n \end{array} \right. \\ \\
& \Longleftrightarrow (a_1,\ldots,a_d) - nv_{J,\underline{s}} \in C_J \cap \mathcal{A}^{(\alpha)}
\end{align*}
Since the one described is a one-to-one correspondence between points in the two sets, our claim is proved. In particular, $\left|\left(C_J + nv_{J,\underline{s}}\right) \cap \mathcal{A}^{(\alpha)} \right|$ is independent of $\underline{s}$. To explicitly express the cardinality of these sets, we note that
\begin{align*}
&C_J \cap \mathcal{A}^{(\alpha)}  = \{(a_1,\ldots,a_d) \in C_J : t_1a_1+t_2a_2 + \ldots + t_da_d \equiv \alpha \mod n\} \\ \\
& =\{(a_1,\ldots,a_d) : \sum_{\ell \in J} t_\ell a_\ell \equiv \alpha -\sum_{\ell \notin J} t_\ell a_\ell \mod n, 0 \leq a_\ell \leq n-1 \mbox{ if } \ell \in J, 0 \leq a_\ell \leq r_e-1 \mbox{ if } \ell \notin J\} \\ \\
&= \bigsqcup_{\tiny{\begin{array}{c} 0 \leq a_\ell \leq r_e-1 \\ \ell \notin J\end{array}}} \{(a_{j_1},\ldots,a_{j_i}) \in ([0,n-1]\cap\N)^i :  \sum_{\ell \in J} t_\ell a_\ell \equiv \alpha -\sum_{\ell \notin J} t_\ell a_\ell \mod n\}.
\end{align*}
Observe that, for a given choice of a $(d-i)$-uple $(a_\ell)_{\ell \notin J}$, the congruence $\sum_{\ell \in J} t_\ell a_\ell \equiv \alpha -\sum_{\ell \notin J} t_\ell a_\ell$ modulo $n$ has a solution $(a_{j_1},\ldots,a_{j_i})$ if and only if $g_J$ divides $\alpha-\sum_{\ell \notin J} a_\ell t_\ell$. In turn, this happens if and only if $(a_\ell)_{\ell \notin J} \in \mathcal{B}_J^{(\alpha)}$, as defined in Notation \ref{notation_psi}. For every such $(a_\ell)_{\ell \notin J} \in \mathcal{B}_J^{(\alpha)}$, we have $g_Jn^{i-1}$ incongruent solution, by Lemma \ref{idolo}. Summing up, we have
\begin{align*}
\ds |C_J \cap \mathcal{A}^{(\alpha)}| & = \sum_{\tiny{\begin{array}{c} \ell \notin J\\ (a_\ell) \in \mathcal{B}_J^{(\alpha)} \end{array}}} \left|\left\{(a_{j_1},\ldots,a_{j_i}) \in ([0,n-1] \cap \N)^i  : \sum_{\ell \in J} t_\ell a_\ell \equiv \alpha-\sum_{\ell \notin J} t_\ell a_\ell \mod n\right\}\right| \\ \\
& = \theta_J^{(\alpha)}(e) \cdot \left|\left\{(a_{j_1},\ldots,a_{j_i}) \in ([0,n-1] \cap \N)^i : \sum_{\ell \in J} t_\ell a_\ell \equiv \alpha -\sum_{\ell \notin J} t_\ell a_\ell \mod n \right\}\right| \\
& = \theta_J^{(\alpha)}(e)\cdot g_Jn^{i-1}.
\end{align*}
\end{proof}
\begin{Remark}
To illustrate the statement of Proposition \ref{counting C_J}, we refer to the specific example of Figure \ref{figure}, in the case $t_1=1$, $t_2=2$, and $\alpha=0$. 
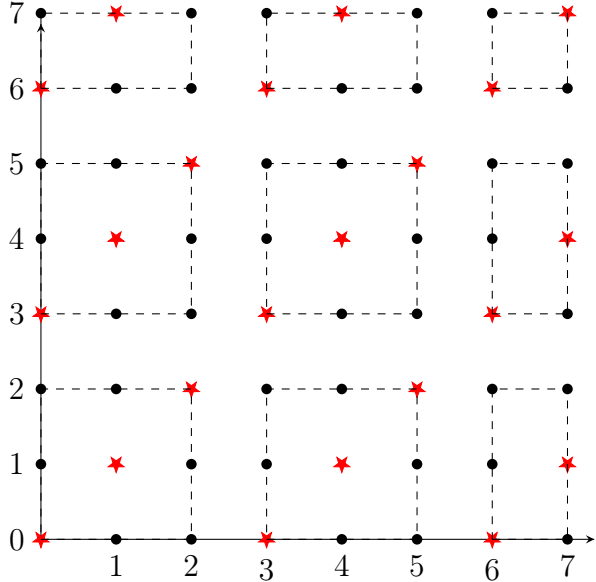
\begin{figure}[h]
\begin{tikzpicture}
  [decoration={markings,mark=at position 1 with {\arrow{stealth}}},
   blob/.style={circle,fill,minimum width=4pt,inner sep=0pt},
   flow/.style={postaction={decorate}}
  ]
  \node[draw, fill,star, red, star points=5, star point ratio=.4, minimum size=2.5pt, inner sep=0pt, label=left:0] (O) at (-1, 0) {};
  \node (Y) at (-1, 7) {};
  \node (X) at (6.5,0) {};
  \draw[flow] (O)  -- (Y);
  \draw[flow] (O)  -- (X);
  \node[blob,label=left:1] (y1) at (-1,1) {};
  \node[blob, label=left:2] (y2) at (-1,2) {};
  \node[draw, fill,star, red, star points=5, star point ratio=.4, minimum size=2.5pt, inner sep=0pt, label=left:3] (y3) at (-1,3) {};
  \node[blob, label=left:4] (y4) at (-1,4) {};
  \node[blob, label=left:5] (y5) at (-1,5) {};
  \node[draw, fill,star, red, star points=5, star point ratio=.4, minimum size=2.5pt, inner sep=0pt, label=left:6] (y6) at (-1,6) {};
  \node[blob, label=left:7] (y7) at (-1,7) {};
  \node[blob, label=below:1] (x1) at (0,0) {};
  \node[blob, label=below:2] (x2) at (1,0) {};
  \node[draw, fill,star, red, star points=5, star point ratio=.4, minimum size=2.5pt, inner sep=0pt, label=below:3] (x3) at (2,0) {};
  \node[blob, label=below:4] (x4) at (3,0) {};
  \node[blob, label=below:5] (x5) at (4,0) {};
  \node[draw, fill,star, red, star points=5, star point ratio=.4, minimum size=2.5pt, inner sep=0pt, label=below:6] (x6) at (5,0) {};
  \node[blob, label=below:7] (x7) at (6,0) {};
  \node[draw, fill,star, red, star points=5, star point ratio=.4, minimum size=2.5pt, inner sep=0pt] () at (0,1) {};
  \node[blob] () at (1,1) {};
  \node[blob] () at (2,1) {};
  \node[draw, fill,star, red, star points=5, star point ratio=.4, minimum size=2.5pt, inner sep=0pt] () at (3,1) {};
  \node[blob] () at (4,1) {};
  \node[blob] () at (5,1) {};
  \node[draw, fill,star, red, star points=5, star point ratio=.4, minimum size=2.5pt, inner sep=0pt] () at (6,1) {};
  \node[blob] () at (0,2) {};
  \node[draw, fill,star, red, star points=5, star point ratio=.4, minimum size=2.5pt, inner sep=0pt] () at (1,2) {};
  \node[blob] () at (2,2) {};
  \node[blob] () at (3,2) {};
  \node[draw, fill,star, red, star points=5, star point ratio=.4, minimum size=2.5pt, inner sep=0pt] () at (4,2) {};
  \node[blob] () at (5,2) {};
  \node[blob] () at (6,2) {};
  \node[blob] () at (0,3) {};
  \node[blob] () at (1,3) {};
  \node[draw, fill,star, red, star points=5, star point ratio=.4, minimum size=2.5pt, inner sep=0pt] () at (2,3) {};
  \node[blob] () at (3,3) {};
  \node[blob] () at (4,3) {};
  \node[draw, fill,star, red, star points=5, star point ratio=.4, minimum size=2.5pt, inner sep=0pt] () at (5,3) {};
  \node[blob] () at (6,3) {};
  \node[draw, fill,star, red, star points=5, star point ratio=.4, minimum size=2.5pt, inner sep=0pt] () at (0,4) {};
  \node[blob] () at (1,4) {};
  \node[blob] () at (2,4) {};
  \node[draw, fill,star, red, star points=5, star point ratio=.4, minimum size=2.5pt, inner sep=0pt] () at (3,4) {};
  \node[blob] () at (4,4) {};
  \node[blob] () at (5,4) {};
  \node[draw, fill,star, red, star points=5, star point ratio=.4, minimum size=2.5pt, inner sep=0pt] () at (6,4) {};
  \node[blob] () at (0,5) {};
  \node[draw, fill,star, red, star points=5, star point ratio=.4, minimum size=2.5pt, inner sep=0pt] () at (1,5) {};
  \node[blob] () at (2,5) {};
  \node[blob] () at (3,5) {};
  \node[draw, fill,star, red, star points=5, star point ratio=.4, minimum size=2.5pt, inner sep=0pt] () at (4,5) {};
  \node[blob] () at (5,5) {};
  \node[blob] () at (6,5) {};
  \node[blob] () at (0,6) {};
  \node[blob] () at (1,6) {};
  \node[draw, fill,star, red, star points=5, star point ratio=.4, minimum size=2.5pt, inner sep=0pt] () at (2,6) {};
  \node[blob] () at (3,6) {};
  \node[blob] () at (4,6) {};
  \node[draw, fill,star, red, star points=5, star point ratio=.4, minimum size=2.5pt, inner sep=0pt] () at (5,6) {};
  \node[blob] () at (6,6) {};
  \node[draw, fill,star, red, star points=5, star point ratio=.4, minimum size=2.5pt, inner sep=0pt] () at (0,7) {};
  \node[blob] () at (1,7) {};
  \node[blob] () at (2,7) {};
  \node[draw, fill,star, red, star points=5, star point ratio=.4, minimum size=2.5pt, inner sep=0pt] () at (3,7) {};
  \node[blob] () at (4,7) {};
  \node[blob] () at (5,7) {};
  \node[draw, fill,star, red, star points=5, star point ratio=.4, minimum size=2.5pt, inner sep=0pt] () at (6,7) {};
  \draw[dashed] (-1,0) rectangle (1,2);
  \draw[dashed] (-1,3) rectangle (1,5);
  \draw[dashed] (2,0) rectangle (4,2);
  \draw[dashed] (2,3) rectangle (4,5);
  \draw[dashed] (5,0) rectangle (6,2);
  \draw[dashed] (5,3) rectangle (6,5);
  \draw[dashed] (5,6) rectangle (6,7);
  \draw[dashed] (-1,6) rectangle (1,7);
  \draw[dashed] (2,6) rectangle (4,7);
\end{tikzpicture}
\label{figure2}
\end{figure}

The points inside $(C_J + v_{J,\underline{s}})\cap \mathcal{A}^{(0)} $ are depicted as red stars. Observe that, as stated in Proposition \ref{counting C_J}, the number of red stars contained in each $(C_J  + v_{J,\underline{s}})\cap \mathcal{A}^{(0)}$ is the same for every fixed $J$. For example, if $J=[2]$, there are $\theta_{J}^{(0)}(3)\cdot g_J \cdot 3^{2-1} = 3$ red stars in each region.
\end{Remark}

\subsection{F-signature function of cyclic quotient singularities}
Let $S=\kk\ps{x_1,\ldots,x_d}$, where $\kk$ is an algebraically closed field of characteristic $p>0$. Let $G$ be a finite small cyclic group of order $n$, with $p$ that does not divide $n$, and $R = S^G$ be the ring of invariants under the action of $G$. Given that $\rank_R(M_\alpha) = 1$ for all $0 \leq \alpha \leq n-1$ by Remark \ref{Remark_irreps_cyclic}, Theorem \ref{theorem-Fsignaturefunctionquotient} allows us to write the multiplicity functions as follows:
\[
\ds \mult(M_\alpha,R^{1/p^e}) =  \frac{p^{de}}{n} + \varphi_{d-2}^{(\alpha)} p^{(d-2)e} + \ldots + \varphi_1^{(\alpha)} p^e + \varphi_0^{(\alpha)},
\]
where the functions $\varphi_c$ are bounded and periodic. The main goal of this section is to give a more explicit description of the functions $\varphi_c$ in case $G$ is cyclic. To achieve this goal, we combine the results we obtained in Section \ref{Section_quotient_sing} and Subsection \ref{Subsection_congruence}.
\begin{Theorem}\label{theorem-Fsignaturecyclic} Let $S=\kk\ps{x_1,\ldots,x_d}$, where $\kk$ is an algebraically closed field of characteristic $p>0$. Let $G$ be a finite small cyclic group of order $n$, with $p$ that does not divide $n$, and $R = S^G$ be a $\frac{1}{n}(t_1,\ldots,t_d)$ cyclic quotient singularity. For all $e\in\mathbb{N}$, write $p^e=kn+r_e$, where $0 < r_e < n$. With the notation introduced in \ref{notation_psi}, for $e \in \N$ we have
\[
\ds  \varphi_c^{(\alpha)}(e) = \frac{1}{n} \left[\sum_{i=c}^d (-1)^{i-c}{i \choose c}\psi_i^{(\alpha)}r_e^{i-c}\right].
\]
\end{Theorem}
\begin{proof}
Combining Proposition \ref{prop-Fsignatureofcyclic}, Lemma \ref{partition} and Proposition \ref{counting C_J} we see that
\begin{align*}
\ds \mult(M_\alpha,R^{1/p^e}) & = |[0,p^e-1]^d \cap \mathcal{A}^{(\alpha)}| \\ \\
& = \left| \bigsqcup_{i=0}^d \bigsqcup_{J \in \Gamma_i} \left(\bigsqcup_{\tiny{ \underline{s} \in ([0,k-1] \cap \N)^i}} ((C_J + nv_{J,\underline{s}}) \cap \mathcal{A}^{(\alpha)}) \right)\right| \\ \\
& = \sum_{i=0}^d \sum_{J \in \Gamma_i} \left(\sum_{\tiny{ \underline{s} \in ([0,k-1] \cap \N)^i}} |(C_J + nv_{J,\underline{s}}) \cap \mathcal{A}^{(\alpha)}| \right) \\
& = \sum_{i=0}^d \sum_{J \in \Gamma_i} k^i|C_J \cap \mathcal{A}^{(\alpha)}| & \mbox{ by Proposition \ref{counting C_J}} \\ \\
& = \sum_{i=0}^d \sum_{J \in \Gamma_i} k^i\theta_J^{(\alpha)}g_Jn^{i-1} & \mbox{ by Proposition \ref{counting C_J}} \\ \\
& = \sum_{i=0}^d  k^in^{i-1} \psi_i^{(\alpha)}.
\end{align*}
Now recall that $k=\frac{p^e-r_e}{n}$, so that $k^i = \sum_{c=0}^i (-1)^{i-c}{i \choose c} p^{ce}r_e^{i-c}$. Substituting this into the formula gives
\begin{align*}
\ds \mult(M_\alpha,R^{1/p^e}) & = \sum_{i=0}^d \left(\frac{p^e-r_e}{n}\right)^i n^{i-1} \psi_i^{(\alpha)}\\ \\
& = \frac{1}{n}\sum_{i=0}^d\sum_{c=0}^i(-1)^{i-c}{i \choose c}\psi_i^{(\alpha)}r_e^{i-c}p^{ce} \\ \\
& =  \sum_{c=0}^d\frac{1}{n}\left[\sum_{i=c}^d (-1)^{i-c}{i \choose c}\psi_i^{(\alpha)}r_e^{i-c} \right]p^{ce} & \mbox{ changing the order of the sum}.
\end{align*}
From this expression, if follows that  $\varphi_c^{(\alpha)} = \frac{1}{n}\left[\sum_{i=c}^d (-1)^{i-c}{i \choose c}\psi_i^{(\alpha)}r_e^{i-c} \right]$, as desired.
\end{proof}

More generally, we have seen in Section \ref{Section_quotient_sing} that the existence of $c$-pseudoreflections inside $G$ determines the vanishing of the higher coefficients of $\mult(M_\alpha,R^{1/p^e})$. In the case of cyclic quotient singularities, we can relate this fact to the values $g_J$, for $J \in \Gamma_c$. We prove this fact in Theorem \ref{theorem-peggiodelcasomonomiale}.
Before that, we need the following lemma which follows from well-known identities between binomial coefficients.

\begin{Lemma} \label{lemma2} Given integers $0 \leq c \leq d-1$, we have
\[
\ds \sum_{i=c}^d (-1)^{i-c} {d \choose i}{i \choose c} = 0.
\]
\end{Lemma}

The following Theorem can be viewed as an improvement of Proposition \ref{prop-Fsignaturefuncionquotient}. Recall that $\theta_J$, $\varphi_i$ and $\psi_i$ denote the functions $\theta_J^{(0)}$, $\varphi_i^{(0)}$ and $\psi_i^{(0)}$ , respectively.

\begin{Theorem}\label{theorem-peggiodelcasomonomiale}
With the notation of Theorem \ref{theorem-Fsignaturecyclic}, consider an integer $1 \leq c \leq d-1$. Then the functions $\varphi_{d-1},\ldots,\varphi_{c}$ are identically zero if and only if  $g_J = 1$ for all $J \in \Gamma_{c}$. Moreover, if $\varphi_\ell$ is the first non-vanishing coefficient with $0 \leq \ell < d$, and $p^e=kn+r_e$, then 
\[
\ds \varphi_\ell = \frac{-{d \choose \ell}r_e^{d-\ell} + \psi_\ell}{n}.
\]
\end{Theorem}
 
\begin{proof}
Assume that $g_J = 1$ for all $J \in \Gamma_{c}$. This implies that $g_J=1$ for all $J \in \Gamma_i$ and $c \leq i \leq d$. It is then easy to see that there are no $i$-pseudoreflections for all $c \leq i \leq d-1$. By Theorem \ref{theorem-Fsignaturefunctionquotient} we conclude that $\varphi_i=0$ for all $c \leq i \leq d-1$.

We now prove the converse. Fix $e\in\mathbb{N}$ such that $r_e=1$. For such a value of $e$, by Theorem \ref{theorem-Fsignaturecyclic} we can express all the coefficients $\varphi_c$ as follows:
\begin{align*}
\ds \varphi_{c}(e) & = \frac{1}{n} \left[\sum_{i=c}^d (-1)^{i-c}{i \choose c}\psi_i(e)\right].
\end{align*} 
In addition, again because $r_e=1$, Remark \ref{rem_psi} gives that $\theta_J(e) = 1$ for all $J \subseteq [d]$ and $e \in \N$. It follows that, $\psi_i(e)= \sum_{J \in \Gamma_i} g_J$ for all $0 \leq i \leq d$. 
Observe that, for every $i$, we have $|\Gamma_i| = {d \choose i}$, and $g_J \geq 1$ for all $J \in \Gamma_i$. Therefore, we always have an inequality $\psi_i \geq {d \choose i}$, with equality that holds if and only if $g_J=1$ for all $J \in \Gamma_i$.

Note that $g_J=1$ for all $J \in \Gamma_{d-1}$, since $G$ is assumed to be small. This will be the base case of our induction. Now let $d-2 \geq c \geq 1$, and assume that $g_J = 1$ for all $J \in \Gamma_{c+1}$. Our previous observation implies that $\psi_i={d \choose i}$ for all $i \geq c+1$. The formula for $\varphi_c$ now gives 
\begin{align*}
\ds \varphi_{c}(e) & = \frac{1}{n}\left[\sum_{i=c}^d (-1)^{i-c}{i \choose c}\psi_i\right] \\ \\
& = \frac{1}{n}\left[\sum_{i=c+1}^d (-1)^{i-c}{i \choose c}{d \choose i} + \psi_{c} \right] \\ \\
& = \frac{1}{n} \left[- {d \choose c}+\psi_{c}\right] & \mbox{ by Lemma \ref{lemma2}.}
\end{align*}
Since $\varphi_{c}(e)=0$ by assumption, we conclude that $\psi_{c} = {d \choose c}$ and, using again the observation made above, we conclude that $g_J=1$ for all $J \in \Gamma_c$, as desired.

For the last part of the theorem, let $\varphi_\ell$ be the first non-zero coefficient, with $0 \leq \ell <d$. By what shown above, we have that $g_J = 1$ for all $J \in \Gamma_{\ell+1}$, and then $\theta_J(e) = r_e^{d-i}$ for all $J \in \Gamma_i$ with $d \geq i \geq \ell+1$ and $e \in \N$, by Remark \ref{rem_psi}. It follows that $\psi_i = {d \choose i}r_e^{d-i}$, again for $d \geq i \geq \ell+1$ and $e \in \N$. By Theorem \ref{theorem-Fsignaturecyclic}, for all $e \in \N$ we finally have that
\begin{align*}
\ds \varphi_{\ell}(e) & = \frac{1}{n}\left[\sum_{i=\ell}^d (-1)^{i-\ell}{i \choose c}\psi_i\right] \\ \\
& = \frac{1}{n}\left[\sum_{i=\ell+1}^d (-1)^{i-\ell}{i \choose \ell}{d \choose i}r_e^{d-i}r_e^{i-\ell} + \psi_{\ell} \right] \\ \\
& = \frac{1}{n}\left[r_e^{d-\ell}\sum_{i=\ell+1}^d (-1)^{i-\ell}{i \choose \ell}{d \choose i} + \psi_{\ell} \right] \\ \\
& = \frac{1}{n} \left[- {d \choose \ell}r_e^{d-\ell}+\psi_{\ell}\right] & \mbox{ by Lemma \ref{lemma2}.}
\end{align*}
\end{proof}

Proposition \ref{prop-Fsignaturefuncionquotient} shows that $\varphi_c=0$ for some $0 \leq c \leq d-2$ implies that $G$ contains no $c$-pseudoreflections. However, it is not true that if $\varphi_c=0$ for one single such $c$, then $g_J=1$ for all $J\in \Gamma_c$. Consider the following example.

\begin{Example} \label{ex-pseudoreflectionsVSgJ=1}
We fix $n=st$, where $s,t>1$ are integers, and consider the $\frac{1}{n}(1,1,t,t)$-cyclic singularity $R$ over an algebraically closed field $\kk$ of characteristic $p\nmid n$. Clearly, we have $g_{\{t\}}=t>1$. We show that the coefficient $\varphi_1$ of $p^e$ in the F-signature function of $R$ is $0$.
This also follows from Theorem \ref{theorem-Fsignaturefunctionquotient}, since the cyclic group $G=\frac{1}{n}(1,1,t,t)$ does not contain $1$-pseudoreflections.
	We show it using the formula 
	\[
	\ds \varphi_1 = \frac{1}{n} \left[\sum_{i=1}^d (-1)^{i-1}{i \choose 1}\psi_ir_e^{i-1}\right]
	\]
	given by Theorem \ref{theorem-Fsignaturecyclic}. 
Let $p^e=kn+r_e$, with $0< r_e<n$.
Since $g_{J}=1$ for $J\in\Gamma_4$ and $J\in\Gamma_3$, we have $\psi_4=1$ and $\psi_3=4r_e$.
	For $j=2$, we have $g_{\{1,1\}}=g_{\{1,t\}}=1$, and $\theta_{\{1,1\}}=\theta_{\{1,t\}}=r_e^2$, so $\psi_2=5r_e^3+r_e\theta_{\{t,t\}}g_{\{t,t\}}$.
	For $j=1$, we have 	$g_{\{1\}}=1$, and $\theta_{\{1\}}=r_e^3$, so $\psi_1=2r_e^3+2\theta_{\{t\}}g_{\{t\}}$.
	Now, observe that $\theta_{\{t\}}$ counts the number of triples $a_1,a_2,a_3\in\{0,\dots,r_e-1\}$ such that $a_1+a_2+ta_3\equiv 0$ modulo $g_{\{t\}}=t$. This is $r_e$-times the number of couples $a_1,a_2\in\{0,\dots,r_e-1\}$ such that $a_1+a_2\equiv 0$ modulo $t=g_{\{t,t\}}$, that is, $\theta_{\{t,t\}}$. Thus, we obtain $\theta_{\{t\}}=r_e\theta_{\{t,t\}}$. Finally, the coefficient of $p^e$ in the F-signature function of $R$ is
	\[\begin{split}
	\varphi_1 &= \frac{1}{n} \left[\sum_{i=1}^d (-1)^{i-1}{i \choose 1}\psi_ir_e^{i-1}\right]=\frac{1}{n}\left[-4r_e^3+12r_e^3-10r_e^3-2r_e\theta_{\{t,t\}}t+2r_e^3+2\theta_{\{t\}}t\right]\\
	&= \frac{1}{n}\left[-2r_e\theta_{\{t,t\}}+2r_e\theta_{\{t,t\}}\right]=0.
	\end{split}
	\]
	However, notice that in this case $\varphi_2\neq0$, as $G$ contains a $2$-pseudoreflection.
	For example, choose $n=6$, $t=3$, and $p\equiv 1$ modulo $6$, so that $r_e=1$ for all $e\in\mathbb{N}$. Then the F-signature function of the $\frac{1}{6}(1,1,3,3)$-singularity is the polynomial in $p^e$ given by $FS(e)=\frac{1}{6}p^{4e}+\frac{1}{3}p^{2e}+\frac{1}{2}$.
	
\end{Example}

Theorem \ref{theorem-peggiodelcasomonomiale} relates the vanishing of the coefficients $\varphi_c=\varphi_c^{(0)}$ of $\mult(R,R^{1/p^e})$ to the invariants $g_J$ of the group $G$. Since Theorem \ref{theorem-Fsignaturecyclic} gives analogous formulas for $\mult(M_\alpha,R^{1/p^e})$ when $\alpha \ne 0$, one may expect that similar considerations about the vanishing coefficients $\varphi_c^{(\alpha)}$ may hold true. It turns out that the vanishing of a coefficient $\varphi_c^{(\alpha)}$ for $\alpha \ne 0$ is, in general a weaker condition than the vanishing of $\varphi_c^{(0)}$. In fact, even the vanishing of all the coefficients $\varphi_i^{(\alpha)}$ for $c \leq i \leq d-1$ does not imply that $G$ has no $c$-pseudoreflections.

To better illustrate what can be said in this direction, consider the following conditions:
\begin{enumerate}
\item $g_J = 1$ for all $J \in \Gamma_c$.
\item $G$ does not have any $i$-pseudoreflections (that is, $G_i=\emptyset$) for all $c \leq i \leq d-1$.
\item The function $\varphi_i^{(0)}$ is identically zero for all $c \leq i \leq d-1$.
\item The function $\varphi_i^{(\alpha)}$ is identically zero for all $c \leq i \leq d-1$ and all $0 \leq \alpha \leq n-1$.
\item The function $\varphi_i^{(\alpha)}$ is identically zero for all $c \leq i \leq d-1$ and some $0 \leq \alpha \leq n-1$.
\end{enumerate}

Our previous results show that the first four conditions are equivalent, and clearly (4) implies (5). However, (5) does not imply (1) -- (4), as the following example shows.
\begin{Example} \label{ex-vanishingalpha}
Let $R$ be the $\frac{1}{6}(1,2,3)$ cyclic quotient singularity over an algebraically closed field $\kk$ of characteristic $p \equiv 1$ modulo $6$. Observe that $r_e=1$ for all $e\in\mathbb{N}$, hence the functions $e \mapsto \mult(M_\alpha,R^{1/p^e})$ will actually be polynomials in $p^e$. Using Theorem \ref{theorem-Fsignaturecyclic}, one can compute
\[
\ds \mult(M_\alpha,R^{1/p^e}) = \left\{\begin{array}{ll} \frac{p^{3e}}{6} + \frac{p^e}{2} + \frac{1}{3} & \mbox{ if } \alpha=0 \\ \\
\frac{p^{3e}}{6} - \frac{p^e}{3} + \frac{1}{6}  & \mbox{ if } \alpha=1,5 \\ \\
\frac{p^{3e}}{6} - \frac{1}{6} & \mbox{ if } \alpha=2, 4 \\ \\ 
\frac{p^{3e}}{6} + \frac{p^e}{6} - \frac{1}{3}  & \mbox{ if } \alpha=3
\end{array}
\right.
\]
In particular, since $\varphi_1^{(2)} = \varphi_2^{(2)}=0$ but $\varphi_1^{(0)} \ne 0$, this shows that (5) does not imply (3).
\end{Example}

\begin{Corollary} \label{coroll_gcd1} Let $R$ be a $\frac{1}{n}(t_1,\ldots,t_d)$-cyclic singularity over an algebraically closed field of characteristic $p>0$. If $g_J = \gcd(t_\ell,n) = 1$ for all $1 \leq \ell \leq d$ then, with $p^e=kn+r_e$, the multiplicity functions $e \mapsto \mult(M_\alpha,R^{1/p^e})$ can be written in the form
\[
\ds \mult(M_\alpha, R^{1/p^e}) = \frac{p^{de}-r_e^d}{n} + \theta_{\emptyset}^{(\alpha)}(e),
\]
where $\theta_{\emptyset}^{(\alpha)}(e)=\left|\left\{(a_1,\dots,a_d)\in([0,r_e-1] \cap \N)^d: \ t_1a_1+\cdots+t_da_d \equiv \alpha \mod n \right\}\right|$.
In particular, this applies to the case of a Veronese rings, which correspond to the choice $t_\ell=1$ for all $\ell$.
\end{Corollary}
\begin{Remark} As already noted in Remark \ref{rem_graded} for the results of Section \ref{Section_quotient_sing}, 
analogous versions of Theorems \ref{theorem-Fsignaturecyclic} and \ref{theorem-peggiodelcasomonomiale}, as well as of Corollary \ref{coroll_gcd1}, hold in the graded setup.
\end{Remark}

\section{Examples}\label{section:examples}
In this section, we present several examples in order to show how our results can be used to compute the F-signature function of specific quotient singularities.

\begin{Example}[singularity $E_6$] \label{Ex_E6}
	Let $\kk$ be an algebraically closed field with $\chara\kk=p\neq2,3$. The binary tetrahedral group $BT$ of $\kk$ is the subgroup of $\mathrm{Sl}(2,\kk)$ of order $24$ generated by the matrices
	\begin{equation*}
	A=\begin{pmatrix}
	i_{\kk} & 0 \\ 0 & i_{\kk}^3
	\end{pmatrix}, \ 
	B=\begin{pmatrix}
	0 & i_{\kk} \\ i_{\kk} & 0
	\end{pmatrix}, \
	C=\frac{1}{\sqrt 2}\begin{pmatrix}
	\xi_{\kk} & \xi_{\kk}^3 \\ \xi_{\kk} & \xi_{\kk}^7
	\end{pmatrix}, 
	\end{equation*}
	where $\sqrt 2$ denotes a square root of $2$ in $\kk$, 	$i_{\kk}$ is a primitive $4$-th root of $1$, and $\xi_{\kk}$ is a primitive $8$-th root of $1$.
	The quotient singularity $R=\kk\llbracket u,v \rrbracket^{BT}$ is called $E_6$ singularity, and is isomorphic to the hypersurface $\kk\llbracket x,y,z \rrbracket/(x^2+y^3+z^4)$. 
	We compute its signature function using Theorem \ref{theorem-Fsignaturefunctionquotient}.
	Fix $e\in\mathbb{N}$. 
	The group $BT$ consists of one $2$-pseudoreflection (the identity matrix $I$) and $23$ $0$-pseudoreflections. 
	Therefore, we only need to compute
	\begin{equation}\label{eq-phi0singularityE6}
	\varphi_0(e)=\frac{1}{24}\sum_{g\neq I}\sum_{0\leq a,b<p^e}(\xi_{g,e,1})^{a}(\xi_{g,e,2})^{b},
	\end{equation}
	where  $\xi_{g,e,i}=\phi((\lambda_{g,i})^{1/p^e})\in\mathbb{C}$ and $\lambda_{g,1},\lambda_{g,2}\in\kk$ are the eigenvalues of $g\in BT$.
	Now, observe that
	\begin{enumerate}
		\item since $BT\subseteq\mathrm{Sl}(2,\kk)$ we have $\lambda_{g,2}=\lambda_{g,1}^{-1}$ for all $g\in BT$;
		\item two conjugate matrices have the same eigenvalues;
		\item $(-)^{1/p^e}$ and $\phi$ are group homomorphisms, therefore $\xi_{g,e,i}$ is a root of unity in $\mathbb{C}$ of the same order of $\lambda_{g,i}$ which is the same order of $g$.
	\end{enumerate}
	So we can split the sum over the elements of the group in \eqref{eq-phi0singularityE6} by conjugacy classes.
	In particular, $BT$ has one element conjugate to $-I$, $6$ elements conjugate to $B$, $4$ elements conjugate to $C$, $4$ to $C^2$, $4$ to $C^4$, and $4$ to $C^5$.
	Thus, we can rewrite \eqref{eq-phi0singularityE6} as
	\begin{equation*}
	\varphi_0(e)=\frac{1}{24}\sum_{0\leq a,b<p^e}\left((-1)^a(-1)^b+6i^a(-i)^b+8\eta^a\eta^{-b}+8\eta^{2a}\eta^{-2b}\right),
	\end{equation*}
	where $i\in\mathbb{C}$ and $\eta\in\mathbb{C}$ is a primitive $6$-th root of $1$.
	Notice that $\varphi_0(e)=\varphi_0(e_1)$ if $p^e\equiv p^{e_1}$ modulo $12$.
	Since $\gcd(p,24)=1$, the only possible values of $p^e$ modulo $12$ are $1,5,7$ and $11$.
	It is straightforward to check that for these values we have always $\varphi_0(e)=\frac{23}{24}$.
	Therefore, the F-signature function of the $E_6$ singularity is
	\begin{equation*}
	FS(e)=\frac{1}{24}p^{2e}+\frac{23}{24},
	\end{equation*}
	in accordance with Brinkmann's result \cite{Brinkmann}.
	In a similar way, one may compute the F-signature function of the quotient singularities $E_7$ and $E_8$.
\end{Example}

\begin{Example}[$3$-rd Veronese subring of the singularity $D_4$] \label{Ex_3-VeroneseD6}
	Let $\kk$ be an algebraically closed field with $\chara\kk=p\neq2,3$.
	We consider the group $G$ obtained as extension of the binary dihedral group $BD_2$ generated by the matrices $A$ and $B$ of Example \ref{Ex_E6} by the cyclic group $C_3$ of order $3$ generated by the matrix $\mathrm{diag}(\omega_{\kk},\omega_{\kk})$, where $\omega_{\kk}\in\kk$ is a primitive $3$-rd root of unity.
	In other words, we have a short exact sequence of finite groups
	\begin{equation*}
	1\rightarrow BD_2\rightarrow G\rightarrow C_3\rightarrow 1.
	\end{equation*}
	We can describe this group as $G=\{M\cdot N:\ M\in BD_2, \ N\in C_3\}$.
	In particular, it follows that since $BD_2$ has order $8$, $G$ has order $24$.
	Notice, however that $G$ is not isomorphic to the group $BT$ of Example \ref{Ex_E6}, since for example it contains an element of order $12$, while $BT$ does not.
	The corresponding quotient singularity  $R=\kk\llbracket u,v \rrbracket^{G}\cong(\kk\llbracket u,v\rrbracket^{BD_2})^{C_3}$ can be seen as a $3$-rd Veronese subring of the Kleinian singularity $D_4$. 
	More explicitly, we have $R=\kk\llbracket u^{12} + v^{12}, u^6v^6, u^{15}v^3 - u^3v^{15}\rrbracket$.
	We compute the F-signature function of $R$ using Theorem \ref{theorem-Fsignaturefunctionquotient}.
	We proceed as in Example \ref{Ex_E6}, and we obtain that if $ p\equiv 1\mod 6$ then 
	\begin{equation*}
	FS(e)=\frac{1}{24}p^{2e}+\frac{23}{24},
	\end{equation*}
	 and if  $ p\equiv 5\mod 6$ then
	\begin{equation*}
	FS(e)=\begin{cases}
	\frac{1}{24}p^{2e}+\frac{23}{24} \ \text{ for } e \text{ even},\\ \\
	\frac{1}{24}p^{2e}-\frac{1}{24} \ \text{ for } e \text{ odd}.
	\end{cases}
	\end{equation*}
\end{Example}

The following three examples are explicit applications of Corollary \ref{coroll_gcd1}.

\begin{Example}[$2$-dimensional Veronese ring]\label{Ex_2Veronese}
Let $R=\kk\ps{ x^n,x^{n-1}y,\cdots,xy^{n-1},y^n}$ be the $2$-dimensional $n$-th Veronese ring, with $n\geq 2$. In the notation of Section \ref{section:cyclicquotient}, this corresponds to the $\frac{1}{n}(1,1)$ cyclic quotient singularity. By direct computation, one can see that
\[
\theta_{\emptyset}^{(0)}(e)=\left|\left\{(a,b)\in([0,r_q-1] \cap \N)^2: \ a+b \equiv 0 \mod n \right\}\right|=\max\{1,2r_e-n+1\}.
\]
By Corollary \ref{coroll_gcd1}, the F-signature function of $R$ is then
\[
FS(e)=\frac{p^{2e}-r_e^2}{n}+\max\{1,2r_e-n+1\}.
\]
\end{Example}

\begin{Example}[singularity $A_{n-1}$] \label{Ex_A_{n-1}}
For $n \geq 2$, let $R=\kk\ps{x^n,xy,y^n}$ be a $2$-dimensional $A_{n-1}$-type singularity. In our notation, this corresponds to the $\frac{1}{n}(1,n-1)$ cyclic quotient singularity. In this case, we have 
\[
\theta_{\emptyset}^{(0)}(e)=\left|\left\{(a,b)\in([0,r_e-1] \cap \N)^2: \ a-b \equiv 0 \mod n \right\}\right|=r_e.
\]
It follows from Corollary \ref{coroll_gcd1} that the F-signature function of $R$ is 
\[
FS(e)=\frac{p^{2e}-r_e^2}{n}+r_e.
\]
This is in accordance with \cite{Brinkmann} (see also \cite[Example 4.3]{F-finExc}).
\end{Example}

\begin{Example}[$3$-dimensional Veronese ring]\label{Ex_3Veronese}
Let $R$ be the $3$-dimensional $n$-th Veronese ring, with $n\geq 2$. This corresponds to the $\frac{1}{n}(1,1,1)$ cyclic quotient singularity. We have 
\[
\theta_{\emptyset}^{(0)}(e)=\left|\left\{(a,b,c)\in([0,r_e-1] \cap \N)^3: \ a+b+c \equiv 0 \mod n \right\}\right|=\left|A_0\right|+\left|A_1\right|+\left|A_2\right|,
\]	
where $A_i=\left\{(a,b,c)\in([0,r_e-1] \cap \N)^3: \ a+b+c =i\cdot n \right\}$ for $i=0,1,2$.
We have $A_0=\{(0,0,0)\}$, and for $i=1,2$, $A_i\neq\emptyset$ if and only if $3(r_e-1)\geq i\cdot n$.
\\ We assume $3(r_e-1)\geq n$, and we compute $|A_1|$.
The number $|A_1|$ is equal to the number of ways we can place $3(r_e-1)-n$ objects in $3$ boxes, where each box can contain at most $r_e-1$ objects.
If $0\leq 3(r_e-1)-n\leq r_e-1$, this number is $|A_1|={{3r_e-n-1}\choose{2}}$.
If $r_e\leq3(r_e-1)-n\leq 2(r_e-1)$, it is $|A_1|={{3r_e-n-1}\choose{2}}-3{{2r_e-n-1}\choose{2}}$, where we have to subtract configurations where we put more than $r_e-1$ objects in one box.
We cannot have configurations where two boxes contain more than $r_e-1$ objects, since this would imply $3(r_e-1)-n\geq 2(r_e-1)+1$, which is equivalent to $r_e-2\geq n$; a contradiction, since $r_e<n$.
A similar reasoning yields $|A_2|={{3r_e-2n-1}\choose{2}}$ for $3(r_e-1)\geq 2n$.
Note that, in this case, there are no configurations with more than $r_e-1$ objects in one box, since this would mean that $3(r_e-1)-2n\geq r_e$, which is equivalent to $2r_e-2n-3\geq0$, again contradicting that $r_e<n$.
\\ Therefore, it follows from Corollary \ref{coroll_gcd1} that the F-signature function of $R$ is  
\[
FS(e)=\frac{p^{3e}-r_e^3}{n}+1+{{3r_e-n-1}\choose{2}}-3{{2r_e-n-1}\choose{2}}+{{3r_e-2n-1}\choose{2}},
\]
with the convention that a binomial coefficient ${{u}\choose{2}}=0$ whenever $u<2$.	
\end{Example}	
Our results allow us to compute several of examples of interest, such as the examples of Iyama and Yoshino from \cite{IyamaYoshino}. 

\begin{Example}[Iyama-Yoshino's singularities]
If $R$ is the $\frac{1}{3}(1,1,1)$ cyclic quotient singularity, then by Example \ref{Ex_3Veronese} the F-signature function of $R$ is
\[
FS(e) = 
 \frac{p^{3e}-1}{3} + 1  
\]
if $p\equiv1\mod 3$, and
\[
FS(e) = \left\{
\begin{array}{ll} \frac{p^{3e}-1}{3} + 1 & \mbox{ for } e \text{ even } \\ \\ \frac{p^{3e}-8}{3} + 1 & \mbox{ for } e \text{ odd} \end{array} \right. 
\]
if $p\equiv2\mod 3$.
On the other hand, for the cyclic quotient singularity $\frac{1}{2}(1,1,1,1)$, it follows from Corollary \ref{coroll_gcd1} and the fact that $r_e=1$ for all $e$ that the F-signature function is
\[
\ds FS(e) = \frac{p^{4e}-1}{2} + 1.
\]
\end{Example}

\par We conclude the paper providing one final example, the Klein four group embedded in ${\rm SL}(3,\kk)$. For the F-signature function of its ring of invariants, that turns out being rather easy to compute with our techniques, we see an example where the last coefficient $\varphi_0$ can be zero.
\begin{Example}[Klein four group]
Let $\kk$ be a field of prime characteristic $p\geq 3$. The Klein four group $\ZZ/(2) \times \ZZ/(2)$ can be realized as a subgroup of ${\rm SL}(3,\kk)$ with no $2$-pseudoreflections as follows:
\[G= 
\ds \left\{\left(\begin{matrix} 1 && \\ &1& \\ &&1 \end{matrix}\right), \left(\begin{matrix} 1 && \\ &-1& \\ &&-1 \end{matrix}\right),\left(\begin{matrix} -1 && \\ &1& \\ &&-1 \end{matrix}\right),\ \left(\begin{matrix} -1 && \\ &-1& \\ &&1 \end{matrix}\right)\right\},
\]
where the entries which are not listed should be treated as zeros. It can be shown that the ring of invariants under this action of $G$ is isomorphic to $\kk\ps{x^2,y^2,z^2,xyz}$. By Theorem~\ref{theorem-Fsignaturefunctionquotient}, because there are no $0$-pseudoreflections, the only coefficient in the F-signature function that we have to determine is $\varphi_1$. A straightforward computation gives $\varphi_1(e) = 3$ for all $e$, therefore
\[
\ds FS(e) = \frac{p^{3e}}{4} + \frac{3p^e}{4}.
\]
\end{Example}
\section*{Acknowledgments} Part of this work was realized when the authors were guests at the Universit\'{e} de Neuch\^{a}tel, and at the Royal Institute of Technology (KTH), in Stockholm. The authors thank these institutions for their hospitality.
The authors would like to thank Holger Brenner, Jack Jeffries, Yusuke Nakajima, Anurag Singh, Francesco Strazzanti, Peter Symonds, and Deniz Yesilyurt for many discussions and helpful comments.
In particular, we thank them for suggesting several of the examples in Section \ref{section:examples}, and for helping with the related computations.

\bibliographystyle{alpha}
\bibliography{References}
\end{document}